\documentclass[a4paper,reqno, 11pt]{amsart}  
\usepackage[DIV=12, oneside]{typearea}
\usepackage[utf8]{inputenc}
\usepackage[T1]{fontenc}
\usepackage[english]{babel}
\usepackage[centertags]{amsmath}
\usepackage{amstext,amssymb,amsopn,amsthm}
\usepackage{nicefrac, esint}
\usepackage{mathrsfs}
\usepackage{dsfont}
\usepackage{bbm}
\usepackage{thmtools}
\usepackage{graphicx}
\usepackage{parskip}

\usepackage[backgroundcolor=white, bordercolor=blue,
linecolor=blue]{todonotes}
\usepackage[colorlinks=true, linkcolor=blue, citecolor=black]{hyperref}
\usepackage{enumitem}

\addto\extrasenglish{}
\addto\extrasenglish{}

\declaretheorem[name=Theorem, numberwithin=section]{theorem}
\newtheorem{corollary}[theorem]{Corollary}
\newtheorem{lemma}[theorem]{Lemma}

\newtheorem{definition}[theorem]{Definition}
\newtheorem{theo:hoelderprep_assumption}[theorem]{Assumption}
\newtheorem{assumptions}{Assumption}

\newtheorem*{example*}{Example}
\newtheorem*{robremark*}{Robustness remark}

\declaretheoremstyle[bodyfont=\normalfont]{remark-style}

\numberwithin{equation}{section}

\theoremstyle{plain}

\newcommand{\N}{\mathds{N}}
\newcommand{\R}{\mathds{R}}

\newcommand{\Z}{\mathds{Z}}

\newcommand{\Vlrxn}{V^{\mu}(M_{\lambda r}(x_0)\big|\R^d)}

\newcommand{\ma}{\mu_{\text{axes}}}
\newcommand{\na}{\nu_{\text{axes}}}
\newcommand{\amax}{\alam}

\newcommand{\BIGOP}[1]
{
\mathop{\mathchoice%
{\raise-0.22em\hbox{\huge $#1$}}%
{\raise-0.05em\hbox{\Large $#1$}}{\hbox{\large $#1$}}{#1}}}
\newcommand{\bigtimes}{\BIGOP{\times}}
\def\Xint#1{\mathchoice
   {\XXint\displaystyle\textstyle{#1}}%
   {\XXint\textstyle\scriptstyle{#1}}%
   {\XXint\scriptstyle\scriptscriptstyle{#1}}%
   {\XXint\scriptscriptstyle\scriptscriptstyle{#1}}%
   \!\int}
\def\XXint#1#2#3{{\setbox0=\hbox{$#1{#2#3}{\int}$}
     \vcenter{\hbox{$#2#3$}}\kern-.5\wd0}}

\def\dashint{\Xint-}
\newcommand{\BIGboxplus}{\mathop{\mathchoice%
{\raise-0.35em\hbox{\huge $\boxplus$}}%
{\raise-0.15em\hbox{\Large $\boxplus$}}{\hbox{\large $\boxplus$}}{\boxplus}}}

\DeclareMathOperator*{\osc}{osc}

\DeclareMathOperator{\dvg}{div}

\renewcommand{\d}{\textnormal{d}}
\newcommand{\wtu}{\widetilde{u}}
\newcommand{\alam}{\alpha_{\max}}

\begin{document}
\allowdisplaybreaks
 \title{Nonlocal operators with singular anisotropic kernels}

\author{Jamil Chaker}
\author{Moritz Kassmann}

\address{Fakult\"{a}t f\"{u}r Mathematik\\Universit\"{a}t Bielefeld\\Postfach 
100131\\D-33501 Bielefeld}
\address{Fakult\"{a}t f\"{u}r Mathematik\\Universit\"{a}t Bielefeld\\Postfach 
100131\\D-33501 Bielefeld}

\urladdr{www.math.uni-bielefeld.de/$\sim$kassmann}



\email{jchaker@math.uni-bielefeld.de}
\email{moritz.kassmann@uni-bielefeld.de}

\begin{abstract}
We study nonlocal operators acting on functions in the Euclidean space. The 
operators under consideration generate anisotropic jump processes, e.g., a jump 
process that behaves like a stable process in each direction but with a 
different index of stability. Its generator is the sum of one-dimensional 
fractional Laplace operators with different orders of differentiability. We 
study such operators in the general framework of bounded measurable 
coefficients. We prove a weak Harnack inequality and Hölder regularity results 
for solutions to corresponding integro-differential equations. 
\end{abstract}

\maketitle


\section{Introduction}
In this article we study regularity estimates of weak solutions to 
integro-differential equations 
driven by nonlocal operators with anisotropic singular kernels. Since the 
formulation of the main results involves various technical definitions, let us 
first look at a simple case. 

For $\alpha \in (0,2)$, the fractional Laplace operator $-(-\Delta)^{\alpha/2}$ 
can be represented as an integro-differential operator $L: C^\infty_c(\R^d) \to 
C(\R^d) $ in the following form
\begin{align}\label{eq:def-sym-nonl-op}
L v (x) = \int_{\R^d} \big( v(x+h) - v(x) + v(x-h) 
\big) \; \pi(\d h) \qquad (x \in \R^d) \,,
\end{align}
where the Borel measure $\pi(\d h)$ on $\R^d \setminus \{0\}$ is defined 
by $\pi(\d h) = c_{d, \alpha} \frac{\d h}{|h|^{d+\alpha}}$ and $c_{d, 
\alpha}$ is an appropriate positive constant. Due to its behavior with respect 
to integration and scaling, $\pi$ is a stable L\'{e}vy measure. The fractional 
Laplace operator generates a strongly continuous contraction semigroup, which 
corresponds to a stochastic jump process $(X_t)_{t \geq 0}$ in $\R^d$. Given $A 
\subset \R^d$, the quantity $\pi(A)$ describes the expected number of jumps 
$(X_{t} - X_{t-}) \in A$ within the interval $0 \leq t \leq 1$. A second 
representation of $-(-\Delta)^{\alpha/2}$ is given with the help of Fourier 
analysis because $-\mathcal{F}((-\Delta)^{\alpha/2} u)(\xi) = |\xi|^\alpha 
\mathcal{F}(u) (\xi)$. The function $\xi \mapsto \psi(\xi) = |\xi|^\alpha$ is 
called the multiplier of the fractional Laplace operator or symbol of the 
corresponding stable L\'evy process. 

In this article we study a rather general class of anisotropic nonlocal 
operators, which contains as a simple example an operator $L^{\alpha_1, 
\alpha_2}: C^\infty_c(\R^2) \to C(\R)$ as in \eqref{eq:def-sym-nonl-op} with 
the measure $\pi$ being a singular measure defined by

\begin{align}\label{eq:def-pi-sing}
\pi^{\alpha_1, \alpha_2} (\d h) = c_{1,\alpha_1} |h_1|^{-1-\alpha_1} \d h_1 
\delta_{0}(\d h_2) + c_{1,\alpha_2} |h_2|^{-1-\alpha_2} \d h_2  \delta_{0}(\d 
h_1)\,,
\end{align}

where $h=(h_1, h_2)$ and $\alpha_1, \alpha_2 \in (0,2)$. For smooth functions 
$u$ one easily 
computes $\mathcal{F}(L^{\alpha_1, \alpha_2} u)(\xi) = \big( |\xi_1|^{\alpha_1} 
+ |\xi_2|^{\alpha_2} \big) \mathcal{F}(u) (\xi)$. Since the multiplier equals 
$|\xi_1|^{\alpha_1} + |\xi_2|^{\alpha_2}$, one can identify the 
operator $L^{\alpha_1, \alpha_2}$ with $-(-\partial_{11})^{\alpha_1}
-(-\partial_{22})^{\alpha_2}$. The aim of this article is to study such 
operators with bounded measurable coefficients and to establish local 
regularity results such as Hölder regularity results. Our main auxiliary result 
is a weak Harnack inequality.

Let us briefly explain why the weak Harnack inequality is a suitable tool. The 
(strong) Harnack inequality states that there is a positive constant $c$ such 
that for every positive function $u:\R^d \to \R$ satisfying $Lu = 0$ in $B_2$ 
the estimate $u(x) \leq c u(y)$ holds true for all $x,y \in B_1$. The Harnack 
inequality is known to hold true for $L = -(-\Delta)^{\alpha/2}$, the proof 
follows from the explicit computations in \cite{Rie38}. It is known to fail in 
the case of $L^{\alpha, \alpha}: C^\infty_c(\R^2) \to C(\R)$ as in 
\eqref{eq:def-sym-nonl-op} with  the measure $\pi$ being a singular measure 
defined by \eqref{eq:def-pi-sing} with $\alpha_1 = \alpha_2 = \alpha$, cf. 
\cite{BoSz07} for a analysis based proof and \cite{BaCh10} for a proof using 
the 
corresponding jump process. As a consequence of the main result in 
\cite{DyKa15}, the weak Harnack inequality holds true in this 
setting. The main aim of the present work implies that it holds 
true even in the case $\alpha_1 \ne \alpha_2$.

We study regularity of solutions $u:\Omega \to \R$ to nonlocal equations of the 
form $\mathcal{L} u = f$ in $\Omega$, where $\mathcal{L}$ is a nonlocal 
operator of the form
\begin{align}\label{def:L}
\mathcal{L}u(x)=\lim\limits_{\epsilon\to0} \int_{\R^d\setminus 
B_{\epsilon}(x)} (u(y)-u(x))\,\mu(x,\d y) 
\end{align}
and $\Omega\subset\R^d$ is an open and bounded set. The operator is 
determined by a family of measures $(\mu(x,\cdot))_{x\in\R^d}$, which play the 
role of variable coefficients. Note that we will not assume any further 
regularity of $\mu(x, \d y)$ in the first variable than measurability and 
boundedness. Before discussing the precise assumptions on $\mu(x, \d y)$, let 
us define a family of reference measures $\ma(x,\d y)$. Given 
$\alpha_1,\dots,\alpha_d\in(0,2)$, we consider a family of measures 
$(\ma(x,\cdot))_{x\in\R^d}$ on $\R^d$ defined by 

\begin{align} \label{def:mu_axes} 
\ma(x,\d y)=\sum_{k=1}^d \Big( 
\alpha_k(2-\alpha_k) |x_k-y_k|^{-1-\alpha_k}\,  \d y_k\prod_{i\neq 
k}\delta_{\{x_i\}}(\d y_i) \Big).
\end{align}

The family $(\ma(x,\cdot))_{x\in\R^d}$ is stationary in the sense that there is 
a measure $\na(\d h)$ with $\ma(x,A) = \na(A-\{x\})$ for every $x \in \R^d$ and 
every measurable set $A \subset \R^d$. In other words, if one defines an 
operator $\mathcal{L}$ as in \eqref{def:L} with $\mu$ being replaced by 
$\ma$, then the operator is translation invariant. The measure $\ma(x,\cdot)$ 
charges only those sets that intersect one of the lines $\{x+te_k\, \colon\,  
t\in\R\}$, where $k\in\{1,\dots,d\}$. In order to deal with the anisotropy of 
the measures, we consider corresponding rectangles. Set $\alam = 
\max\{\alpha_i : i \in \{1, \ldots, d\} \}$.
\begin{definition}\label{M_r}
For $r>0$ and $x\in\R^d$ we define 
\begin{align*}
M_r(x) =\bigtimes_{k=1}^d 
\left(x_k-r^{\frac{\amax}{\alpha_k}},x_k+r^{\frac{\amax}{ \alpha_k}}\right) 
\quad \text{ and } M_r = M_r(0) \,.
\end{align*}
For $0 < r \leq 1$, the rectangle $M_r(x)$ equals the 
ball $\{ y\in\R^d \colon \mathbbm{d}(x,y)<r \}$ in the metric space 
$(\R^d,\mathbbm{d})$, where the metric $\mathbbm{d}$ is defined as follows:
\begin{equation}\label{metric}
  \mathbbm{d}(x,y)= 
\sup\limits_{k\in\{1,\dots,d\}}\left\{|x_k-y_k|^{\alpha_k/\amax}\mathds{1}_{\{
|x_k-y_k|\leq 1\}}(x,y)+ \mathds{1}_{\{|x_k-y_k|> 1\}}(x,y)\right\} \,.
\end{equation}
\end{definition}

Let us formulate and explain our main assumptions on 
$(\mu(x,\cdot))_{x\in\R^d}$. 
\begin{assumptions}\label{assum:levy-symmetry} We assume 
\begin{align}\label{eq:assum:levy}\tag{A1-a}
\sup\limits_{x\in\R^d} \int_{\R^d} (|x-y|^2\wedge 1) \mu(x,\d y) 
<\infty \,,
\end{align}
and for all measurable sets $A, B \subset \R^d$
\begin{align}\label{eq:assum:symmetry}\tag{A1-b}
\int_{A}\int_{B}\mu(x,\d y)\, \d x = \int_{B}\int_{A}\mu(x,\d 
y)\, \d x \,.
\end{align}
\end{assumptions}

Note that \eqref{eq:assum:levy} is nothing but an uniform L\'evy-integrability 
condition. It allows $\mu(x, A)$ to have some singularity for $x \in 
\overline{A}$. Condition \eqref{eq:assum:symmetry} asks for symmetry of the 
family 
$(\mu(x,\cdot))_{x\in\R^d}$. Examples of $\mu(x, \d y)$ satisfying 
these two conditions are given by  $\ma$ as in \eqref{def:mu_axes} and by 
\[ \mu_1 (x, \d y) = a(x,y) |x-y|^{-d-\alpha} \d y  \,, \] 
where $\alpha \in (0,2)$ and $a(x,y) \in [1,2]$ is a measurable symmetric 
function.

The following assumption is our main assumption. It relates $\mu(x, \d y)$ to 
the reference family $\ma(x, \d y)$. The easiest way to do this would be to 
assume that there is a constant $\Lambda \geq 1 $ such that for every $x \in 
\R^d$ and every nonnegative measurable function $f:\R^d \times \R^d \to \R$ 
\begin{align}
 \Lambda^{-1} \int f(x,y) \ma(x, \d y) \leq \int f(x,y) \mu(x, \d y) \leq 
 \Lambda \int f(x,y) \ma(x, \d y) \,.
\end{align}
We will work under a weaker condition, which appears naturally in our 
framework. For $u,v \in L^2_{loc}(\R^d)$ and $\Omega\subset\R^d$ open 
and bounded, we define 
\[ \mathcal{E}^{\mu}_{\Omega}(u,v)=\int_{\Omega}\int_{\Omega} 
(u(y)-u(x))(v(y)-v(x))\,\mu(x,\d y)\, \d x\]
and $\mathcal{E}^{\mu}(u,v)= \mathcal{E}^{\mu}_{\R^d}(u,v)$ whenever the 
quantities are finite.

\begin{assumptions}\label{assum:comparability}
There is a constant $ \Lambda \geq 1$ such that for $0 < \rho \leq 
1$, $x_0\in M_1$ and $w\in 
L^2_{loc} (\R^d)$
\begin{align}\label{eq:assum:comparability} \tag{A2}
\Lambda^{-1}\mathcal{E}^{\ma}_{M_{\rho}(x_0)}(w,w)\leq 
\mathcal{E}^{\mu}_{M_{\rho}(x_0)}(w,w) \leq 
\Lambda \mathcal{E}^{\ma}_{M_{\rho}(x_0)}(w,w). 
\end{align}
\end{assumptions}

Let us briefly discuss this assumption. Assume $a(x,y) \in [1,2]$ is 
symmetric and $\ma$ is defined as in \eqref{def:mu_axes} with 
respect to some $\alpha_1, \ldots, \alpha_d \in (0,2)$. If we define $\mu_2$ 
by $\mu_2(x, A) = \int_A a(x,y) \ma(x,\d y)$, then $\mu_2$ obviously
satisfies \autoref{assum:comparability}. If $\alpha_1 = \alpha_2 = 
\ldots = \alpha_d = \alpha$, then it is proved in \cite{DyKa15} that $\mu_1$ 
satisfies \eqref{eq:assum:comparability}. Note that comparability of 
the quadratic forms $\mathcal{E}^{\ma}(w,w)$ and $\mathcal{E}^{\mu_1}(w,w)$ 
follows from comparability of the respective multipliers.  

In general, studying \autoref{assum:comparability} is a research project in 
itself. Let us mention one curiosity. Given $x \in 
\R^d$, \autoref{assum:comparability} does not require  $\mu(x,\d y)$ to be 
singular 
with respect to the Lebesgue measure. One can construct an 
absolutely continuous measure $\nu_{\text{cusp}}$ on $\R^d$ such that for 
$\mu_3$ given by $\mu_3(x,A) = \nu_{\text{cusp}}(A - \{x\})$, 
\autoref{assum:comparability} is satisfied. Since computations are rather 
lengthy, they will be carried out in a future work.

We need one more assumption related to cut-off functions, 
\autoref{assum:cutoff} resp. \eqref{eq:assum:cutoff}. Since this assumption is 
not restrictive at all but 
rather technical, we provide it in \autoref{subsec:aux:cutoff}. The 
quadratic forms introduced above relate to integro-differential operators in 
the following way. 
Given a sufficiently nice family of measures $\mu$ (any of $\mu_a$, $\mu_1$, 
$\mu_2$, would do) and sufficiently regular functions $u,v : \R^d \to 
\R$, one has $\mathcal{E}^{\mu}(u,v) =2  \int_{\R^d} \mathcal{L} u(x) v(x) 
\d x$ with $\mathcal{L}$ as in \eqref{def:L}. That is why we will study 
solutions $u$ to $\mathcal{L} u = f$ with the help of bilinear forms like 
$\mathcal{E}^\mu$. In order to do this, we need to define appropriate 
Sobolev-type 
function spaces.

\begin{definition}\label{VHomega}
Let $\Omega\subset\R^d$ open. We define the function spaces
\begin{align} 
  V^{\mu}(\Omega|\R^d)  &= \Big\{ u:\R^d\to\R \text{ measb.} \, \colon \, 
u\bigr|_{\Omega}\in 
L^2(\Omega), (u,u)_{V^{\mu}(\Omega|\R^d)} <\infty\Big\}\,, 
\label{def:vspace} \\
 H^{\mu}_{\Omega}(\R^d) &= \Big\{ u:\R^d\to\R \text{ measb.}  \, \colon \, 
u\equiv 0 \text{ on } 
\R^d\setminus\Omega, \|u\|_{H^{\mu}_{\Omega}(\R^d)}<\infty \Big\}, 
\label{def:hspace}
\end{align}
where
\begin{align*}
(u,v)_{V^{\mu}(\Omega|\R^d)} &= \int_{\Omega}\int_{\R^d} 
(u(x)-u(y))(v(x)-v(y))\, \mu(x,\d y)\, \d x \,, \\
\|u\|_{H^{\mu}_{\Omega}(\R^d)}^2 &= \|u\|_{L^2(\Omega)}^2 + 
\int_{\R^d}\int_{\R^d} (u(y)-u(x))^2\mu(x,\d y)\,\d x \,.
\end{align*}
\end{definition}

The space $V^{\mu}(\Omega|\R^d)$ is a nonlocal analogon of the space 
$H^1(\Omega)$. Fractional regularity is required inside of $\Omega$ 
whereas in $\R^d \setminus \Omega$ only integrability is imposed. The space 
$H^{\mu}_{\Omega}(\R^d)$ is a nonlocal analogon of $H^1_0(\Omega)$. We are now
in a position to formulate our main results:

\begin{theorem}\label{theo:weakharnack} Assume \eqref{eq:assum:levy}, 
\eqref{eq:assum:symmetry}, 
\eqref{eq:assum:comparability} and \eqref{eq:assum:cutoff}.  Let $f\in 
L^q(M_1)$ 
for some $q>\max\{2,\sum_{k=1}^d\frac{1}{\alpha_k}\}$. 
Assume 
$u\in V^{\mu}(M_1\big|\R^d)$, $u\geq 0$ in $M_1$ satisfies
 \begin{align}\label{eq:supersol}
   \mathcal{E}(u,\varphi)\geq (f,\varphi) \quad \text{ for every non-negative } 
\varphi\in H^{\mu}_{M_1}(\R^d).
 \end{align}
 Then there exist $p_0\in(0,1)$, $c_1>0$, independent of $u$, such that
 \[ \inf\limits_{M_{\frac14}}u \geq c_1\left(\dashint_{M_{\frac12}} 
u(x)^{p_0}\, \d x\right)^{1/p_0} - \sup\limits_{x\in M_{\frac{15}{16}}} 
2\int_{\R^d\setminus M_1} u^{-}(z)\mu(x,\d z) - 
\|f\|_{L^{q}(M_{\frac{15}{16}})}. \]
\end{theorem}

As is well known, the weak Harnack inequality implies a decay of 
oscillation- result and Hölder regularity estimates for weak solutions.

\begin{theorem}\label{theo:hoelder}
Assume \eqref{eq:assum:levy}, \eqref{eq:assum:symmetry}, 
\eqref{eq:assum:comparability} and \eqref{eq:assum:cutoff}. Let $f\in 
L^q(M_1)$ 
for some $q>\max\{1,\sum_{k=1}^d \tfrac{1}{\alpha_k}\}.$ 
 Assume $u\in V^{\mu}(M_1\big|\R^d)$ satisfies
 \[ \mathcal{E}(u,\varphi)=(f,\varphi)\quad \text{for every non-negative } 
\varphi\in 
H^{\mu}_{M_1}(\R^d). \]
 Then there are $c_1\geq 1$ and $\delta\in(0,1)$, independent of $u$, such that 
for almost every $x,y\in M_{\frac{1}{2}}$
 \begin{equation}\label{Hoelder-estimate}
  |u(x)-u(y)|\leq c_1|x-y|^{\delta}\left( \|u\|_{\infty} + 
\|f\|_{L^q(M_{\frac{15}{16}})}\right).
 \end{equation}
\end{theorem}

Let us discuss selected related results in the literature. 

The research in this article is strongly influenced by the the fundamental 
contributions of \cite{DEGIORGI, NASH, MOSER}  on H\"older estimates for 
weak solutions $u$ to second order equations of the form

\begin{align}\label{eq:divform-2nd-order}
\dvg(A(x)\nabla u(x))=0 
\end{align}

for uniformly positive definite and measurable coefficients 
$A(\cdot)$. In particular, \cite{MOSER} underlines the 
significance of the Harnack inequality for weak solutions to this equation. 
Note that, similarly to the present work, \eqref{eq:divform-2nd-order} is 
interpreted in the weak sense, i.e. instead of \eqref{eq:divform-2nd-order} one 
assumes

\[ \mathcal{E}^{local}(u,v) := \int A(x)\nabla u(x) \nabla v(x) \d x =0  \]

for every test function $v$. Analogous results for similar integro-differential 
equations with differentiability order $\alpha \in (0,2)$ have been studied by 
several authors and with the help of different methods. Important 
contributions include \cite{BASSLEVTRANS, KUMA, Kas09, KUUSI2, CAFFVASS, 
KASFELS, DyKa15, CHENKUMAWANG, Coz17}. These articles include operators of 
the form \eqref{def:L} with $\mu = \mu_1$ are studied and no further regularity 
assumption in $a(x,y)$ apart from boundedness is assumed. Note that, formally 
speaking, Hölder regularity estimates for fractional equations are stronger than 
the ones for local equations if the results are robust with respect to $\alpha 
\to 2-$ as in \cite{Kas09, Coz17}. Of course, there are many more results 
related to Hölder regularity estimates for solutions to integro-differential 
equations related to energy forms. The aforementioned articles serve as a good 
starting point for further results. Hölder regularity results have also been 
obtained for nonlocal equations in non-divergence form, i.e., for operators not 
generating quadratic forms. 

We comment on related regularity results if the measures are singular with 
respect to the Lebesgue measure. \cite{BaCh10} and \cite{ZHANG15} study 
regularity of solutions to systems of stochastic differential equations which 
lead to nonlocal operators in nondivergence form with singular measures 
including versions of $L^{\alpha, \alpha}$ with continuous bounded coefficients. 
These results have been extended to the case of operators  with possibly 
different values for $\alpha_i$ in \cite{JAMIL}. Assuming that the systems 
studied in \cite{BaCh10} are diagonal, \cite{KuRy17} establishes 
sharp two-sided heat kernel estimates. It is very interesting that operators of 
the form $L^{\alpha_1, \alpha_2}$ appear also in the study of random walks on 
groups driven by anisotropic measures. Results on the potential theory can be 
found in \cite{SaZh13}, \cite{SaZh15}, \cite{SaZh16}. 

The closest to our article is \cite{DyKa15} from which we borrow several ideas. 
\cite{DyKa15} establishes results similar to \autoref{theo:weakharnack} and 
\autoref{theo:hoelder} in a general framework which includes operators 
\eqref{def:L} with $\ma$ and $\mu_2$. The assumption $\alpha_1 = \alpha_2 = 
\ldots = \alpha_d$ is essential for the main results in \cite{DyKa15}. The 
main aim of the present work is to remove this restriction. This makes it 
necessary to study the anisotropic setting in detail and to develop new 
functional inequalities resp. embedding results. Luckily, the John-Nirenberg 
embedding has been established by others in the context of general metric 
measure spaces. Note that, different from \cite{DyKa15}, we allow the functions 
$u$ to be (super-)solutions for inhomogeneous equations. The additional 
right-hand side $f$ does not create substantial difficulties. 

The article is organized as follows. \autoref{sec:aux} contains auxiliary 
results like function inequalities, embedding results, and technical results 
regarding cut-off functions. In \autoref{sec:prop_weak_sol} we establish 
several intermediate results for functions $u$ satisfying 
\eqref{eq:supersol} and prove \autoref{theo:weakharnack}. 
In \autoref{sec:hoelder} we deduce \autoref{theo:hoelder}.

\section{Auxiliary results}\label{sec:aux}

The aim of this section is to provide more or less technical results needed 
later. In particular, we introduce appropriate cut-off functions, establish 
Sobolev-type embeddings and prove a Poincar\'{e} inequality in our anisotropic 
setting. Recall that we work with \autoref{assum:levy-symmetry} and 
\autoref{assum:comparability} in place.

\subsection{Cut-off functions}\label{subsec:aux:cutoff}

As mentioned above, we need to impose one further condition to 
\autoref{assum:levy-symmetry} and \autoref{assum:comparability}. We need to 
make sure that the nonlocal operator $\mathcal{L}$ resp. the quadratic forms 
behave nicely with respect to cut-off functions.  Let us explain a simple 
example first. If $r>0$ and $\tau  \in C^2_c(\overline{B_{2r}})$ with $\tau 
\equiv 1$ on $B_r$ and $\tau$ linear on $B_{2r}\setminus B_r$, then $|\nabla 
\tau| \leq cr^{-1}$ in $\R^d$ with a constant independent of $r$. Let us a 
explain a similar relation in our nonlocal anisotropic setting. Note that, in 
general, the nonlocal analogon of $|\nabla \tau(x)|^2$ is given by $\frac12 
 \int_{\R^d} (\tau(y)-\tau(x))^2 \mu(x,\d y)$. In the framework of Dirichlet 
forms, both objects are the corresponding carr\'{e} du champ operator of 
$\tau$.

Assume, for some $x_0\in M_1$, $r\in(0,1]$, $\lambda>1$, $\tau\in 
C^1(\R^d)$ is a cut-off function satisfying 
\begin{align}\label{def:cutoff_aniso}
 \begin{cases}
  \text{supp}(\tau)\subset M_{\lambda r}(x_0), \\
  \|\tau\|_\infty\leq 1, \\
  \tau\equiv 1 \text{ on } M_r(x_0), \\
  \|\partial_k \tau\|_{\infty}\leq 
\frac{2}{(\lambda^{\amax/\alpha_k}-1)r^{\amax/\alpha_k}}\, \text{ for all } 
k\in\{1,\dots,d\}.
 \end{cases}
\end{align}

\begin{figure}[ht]
 \begin{center}
 \includegraphics[width=0.6\textwidth]{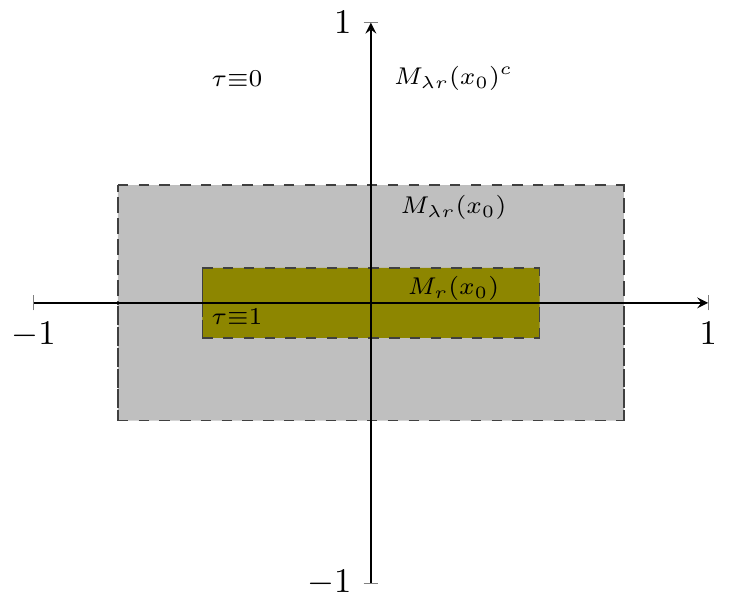} 
  \caption{Example of $\tau$ for $x_0=0$, $\alpha_1=\frac32$, 
$\alpha_2=\frac12$, $r=\frac12$, 
$\lambda=\frac32$}
  \end{center}
 \end{figure}

\begin{lemma}\label{lem:cut-off}
There is a constant $c_1>0$, independent of 
$x_0, \lambda, r, 
\alpha_1,\dots,\alpha_d$ and $\tau$, such that
\[ \sup_{x\in\R^d} \int_{\R^d} (\tau(y)-\tau(x))^2\ma(x,\d y)\leq 
c_1r^{-\amax}\left(\sum_{k=1}^d 
(\lambda^{\amax/\alpha_k}-1)^{-\alpha_k}\right). \]
\end{lemma}
\begin{proof}
 Set 
\[I_k=\left(x_k-(\lambda^{\amax/\alpha_k}-1)r^{\amax/\alpha_k},x_k+(\lambda^{
\amax/\alpha_k}-1)r^{\amax/\alpha_k}\right).\] 
 Then we have for any $x\in\R^d$
 \begin{align*}
  \int_{\R^d}& (\tau(x)-\tau(y))^2\ma(x,\d y)
   \leq \sum_{k=1}^d \left[\|\partial_k \tau \|_{\infty}^2\int_{I_k} 
\frac{\alpha_k(2-\alpha_k)}{|h|^{1+\alpha_k-2}} \,\d h + \int_{\R\setminus I_k} 
\frac{\alpha_k(2-\alpha_k)}{|h|^{1+\alpha_k}} \,\d h\right] \\
 & \leq \sum_{k=1}^d 
\left[4\alpha_k\left((\lambda^{\amax/\alpha_k}-1)r^{\amax/\alpha_k}\right)^{
-\alpha_k} + 2(2-\alpha_k)\left(\lambda^{\amax/\alpha_k}-1)
 r^{\amax/\alpha_k}\right)^{-\alpha_k}\right] \\
  & \leq 8 r^{-\amax}\left(\sum_{k=1}^d 
(\lambda^{\amax/\alpha_k}-1)^{-\alpha_k}\right).
  \end{align*}
 \end{proof}

Given $\tau$ as above, we assume that the nonlocal carr\'{e} du champ 
operator with respect to $\mu$ is uniformly dominated by the one with respect 
to $\ma$. 

\begin{assumptions}\label{assum:cutoff}
Let $x_0\in M_1$, $r\in(0,1]$ and $\lambda>1$. We assume there is a $c_1>0$, 
independent of $x_0, \lambda, r, 
\alpha_1,\dots,\alpha_d$ and $\tau$, such that
\begin{align} \label{eq:assum:cutoff}\tag{A3}
\sup_{x\in\R^d} \int_{\R^d} (\tau(y)-\tau(x))^2\mu(x,\d y) \leq 
c_1\sup_{x\in\R^d} \int_{\R^d} (\tau(y)-\tau(x))^2\ma(x,\d y).
\end{align}
\end{assumptions}

From now on, we assume that the family $\mu(x,\cdot), x \in \R^d$ always 
satisfies \autoref{assum:cutoff}. For future purposes, let us deduce a helpful 
observation. 

\begin{corollary}\label{quadrat}
Let $x_0\in M_1$, $r\in(0,1]$, $\lambda>1$ and $\tau\in C^1(\R^d)$. 
Assume $\tau$ satisfies  \eqref{def:cutoff_aniso}.
There is a constant $c_1>0$, independent of $u, x_0, \lambda, r, 
\alpha_1,\dots,\alpha_d$, such that for any $u\in \Vlrxn$
 \[ \int_{M_{\lambda r}(x_0)}\int_{\R^d\setminus M_{\lambda 
r}(x_0)}u(x)^2\tau(x)^2\, \mu(x,\d y)\, \d x \leq c_1 
r^{-\amax}\left(\sum_{k=1}^d 
(\lambda^{\amax/\alpha_k}-1)^{-\alpha_k}\right)\|u\|_{L^2(M_{\lambda 
r}(x_0)}^2. \]
\end{corollary}

\subsection{Sobolev-type inequalities}

One important tool in our studies will be Sobolev-type inequalities. We begin 
with a comparability result, which gives a representation of 
$(u,u)_{V^{\ma}(\R^d|\R^d)}$ in terms of the Fourier transform of $u$.

\begin{lemma}\label{multiplier}
Let $u\in V^{\ma}(\R^d\big|\R^d)$. Then there is a constant $C>1$ that depends 
only the dimension $d$ such that
 \[ C^{-1}\Big\|\widehat{u}(\xi)\Big(\sum_{k=1}^d 
|\xi_k|^{\alpha_k}\Big)^{\frac{1}{2}} \Big\|_{L^2_{\xi}(\R^d)}^2 \leq 
\mathcal{E}^{\ma}(u,u) 
  \leq C \Big\|\widehat{u}(\xi)\Big(\sum_{k=1}^d 
|\xi_k|^{\alpha_k}\Big)^{\frac{1}{2}} \Big\|_{L^2_{\xi}(\R^d)}^2 \,. \]
\end{lemma}

\begin{proof}
By Fubini's and Plancherel's theorem,
 \begin{align*}
\mathcal{E}^{\ma}(u,u) = \sum_{k=1}^d 
\alpha_k(2-\alpha_k)\int_{\R^d}|\widehat{u}(\xi)|^2\int_{\R}\frac{(1-e^{
i\xi_kh_k})^2}{|h_k|^{1+\alpha_{k}}}\; \d h_k\; \d \xi.
 \end{align*}
Furthermore, there is a constant $c_1\geq 1$, independent of 
$\alpha_1,\dots,\alpha_d$, such that for any $k\in\{1,\dots,d\}$
\[ c_1^{-1}|\xi_k|^{\alpha_k} \leq 
\alpha_k(2-\alpha_k)\int_{\R}\frac{(1-e^{i\xi_kh_k})^2}{|h_k|^{1+\alpha_{k}}}\; 
\d h_k \leq c_1|\xi_k|^{\alpha_k}. \]
Hence the assertion follows.
\end{proof}
One important observation is the following Sobolev-type inequality. We define 
the quantity

\begin{align}\label{beta}
\beta=\sum_{j=1}^d \frac{1}{\alpha_j}.
\end{align} 

\begin{theorem}\label{sobolev-whole}
There is a constant $c_1=c_1(d,2\beta/(\beta-1))>0$ 
such that for every compactly supported $u\in V^{\ma}(\R^d\big|\R^d)$
\[ \|u\|_{L^{\frac{2\beta}{\beta-1}}(\R^d)}^2 \leq c_1 
\left(\int_{\R^d}\int_{\R^d} (u(x)-u(y))^2 \, \ma(x,\d y)\; \d x\right). \]
\end{theorem}

We thank A. Schikorra for discussing this result and its proof with us. We 
believe that this result has been established several times in the literature 
but we were not able to find a reference.

\begin{proof} Let $\Theta:=2\beta/(\beta-1)$. We denote the H\"older conjugate 
of $\Theta$ by $\Theta'$. Note
 \begin{equation}\label{hoelder}
\begin{aligned}
 \|u\|_{L^\Theta(\R^d)} &= \|u\|_{L^{\Theta,\Theta}(\R^d)} \leq 
c_2\|u\|_{L^{\Theta,2}(\R^d)} \leq c_3\|\widehat{u}\|_{L^{\Theta',2}(\R^d)} \\
 & \leq c_3\left\|\left(\sum_{k=1}^d 
|\xi_k|^{\alpha_k}\right)^{-\frac{1}{2}}\right\|_{L_{\xi}^{2\Theta'/(2-\Theta'),
\infty}(\R^d)} \left\|\left(\sum_{k=1}^d 
|\xi_k|^{\alpha_k}\right)^{\frac{1}{2}} 
\widehat{u}(\xi)\right\|_{L_{\xi}^2(\R^d)}.
\end{aligned}
\end{equation}
Our aim is to show 
\[K(\xi)=\left(\sum_{k=1}^d |\xi_k|^{\alpha_k}\right)^{-\frac{1}{2}}\in 
{L^{2\Theta'/(2-\Theta'),\infty}(\R^d)},\] 
which implies the assertion by \autoref{multiplier}.

Let $\xi\in\R^d$. Then there is obviously an index $i\in\{1,\dots,d\}$ such that
\[|\xi_i|^{\alpha_i} \geq |\xi_j|^{\alpha_j}\quad  \text{ for all } j\neq i.\]
Thus there is a $c_4\geq 1$, depending only on $d$, such that
 \[ c_4^{-1}|\xi_i|^{-\alpha_i/2} \leq \left(\sum_{j=1}^d 
|\xi_j|^{\alpha_i}\right)^{-1/2} = \left(|\xi_i|^{\alpha_k}\left(1+\sum_{j\neq 
i} \frac{|\xi_j|^{\alpha_j}}{|\xi_i|^{\alpha_i}}\right)\right)^{-1/2} \leq 
c_4|\xi_i|^{-\alpha_i/2}. \]
 Hence
 \begin{align*}
|\{|K(\xi)\geq t\}| & = \left|\left\{ \left|\left(\sum_{k=1}^d 
|\xi_k|^{\alpha_k}\right)^{-1/2}\right| \geq t \right\}\right|\\
& \leq \sum_{i=1}^d \left|\{(|\xi_i|^{-\alpha_i/2}\geq t) \wedge 
(|\xi_i|^{\alpha_i}\geq |\xi_j|^{\alpha_j}) \text{ for all } j\neq i\}\right|\\
& = \sum_{i=1}^d \left|\{(|\xi_i|\leq t^{-2/\alpha_i}) \wedge (|\xi_j|\leq 
|\xi_i|^{\alpha_i/\alpha_j}) \text{ for all } j\neq i\}\right|=:c_4\sum_{i=1}^d 
\eta_i.
 \end{align*}
For each $i\in\{1,\dots,d\}$, we have

\begin{align*}
 \eta_i & = 2^d \int_{0}^{t^{-2/\alpha_i}} \left(\prod_{j\neq i} 
\int_{0}^{\xi_i^{\alpha_i/\alpha_j}}\, \d \xi_j\right)\d\xi_i  = 2^d 
\int_{0}^{t^{-2/\alpha_i}} \xi_i^{\sum_{j\neq i} \frac{\alpha_i}{\alpha_j}}\, 
\d\xi_i
 = \frac{2^d}{\sum_{j\neq i}\frac{\alpha_i+\alpha_j}{\alpha_j}} 
t^{-\frac{2}{\alpha_i}\left(\sum_{j\neq i}\frac{\alpha_i}{\alpha_j}+1\right)} \\
 & \leq \frac{2^d}{d-1} t^{-2\left(\sum_{j=1}^d \frac{1}{\alpha_j}\right)}=c_5 
t^{-2\beta}.
\end{align*}

Hence, we have $K\in L^{2\beta,\infty}$, if

\begin{align*}
\frac{2\Theta'}{2-\Theta'} = 2\beta \iff \frac{2-\Theta'}{\Theta'} = 
\frac{1}{\beta} \iff \frac{1}{\Theta} = \frac{1}{2} - 
\frac{1}{2\beta}=\frac{1}{2}\left(\frac{\beta-1}{\beta}\right) \iff \Theta = 
\frac{2\beta}{\beta-1},
\end{align*}

from which the assertion follows.
\end{proof}

Note that the case $\alpha_1=\cdots=\alpha_d=\alpha\in(0,2)$ leads to 
$\beta=d/\alpha$ and therefore 
\[\Theta=\frac{2\beta}{\beta-1}=\frac{2d}{d-\alpha}\] 
in  \autoref{sobolev-whole}, which arises in the Sobolev embedding 
$H^{\alpha/2}(\R^d)\subset L^{\Theta}(\R^d)$. 

\begin{theorem}\label{sobolev-domainaxes}
Let $x_0\in M_1$, $r\in(0,1]$ and $\lambda>1$. Let $u\in V^{\ma}(M_{\lambda 
r}(x_0)\big|\R^d)$. 
Then there is a constant $c_1=c_1(d,2\beta/(\beta-1))>0$, independent 
of $x_0, \lambda, r, \alpha_1,\dots,\alpha_d$ and $u$, such that
 \begin{equation}
 \begin{aligned}
  \|u\|_{L^{\frac{2\beta}{\beta-1}}(M_r(x_0))}^2 & \leq c_1 
\Bigg(\int_{M_{\lambda r}(x_0)}\int_{M_{\lambda r}(x_0)} (u(x)-u(y))^2 \, 
\ma(x,\d y)\, \d x\\
  & \qquad \quad  +r^{-\amax}\left(\sum_{k=1}^d 
(\lambda^{\amax/\alpha_k}-1)^{-\alpha_k}\right)\|u\|^2_{L^2(M_{\lambda 
r}(x_0))}\Bigg).  
 \end{aligned}
 \end{equation}
\end{theorem}
\begin{proof}
 Let $\tau:\R^d\to\R$ be as in \eqref{def:cutoff_aniso}. 
 For simplicity of notation we write $M_r=M_r(x_0)$. Let $v\in L^2(\R^d)$ such 
that $v\equiv u$ on $M_{\lambda r}$ and $\mathcal{E}(v,v)<\infty$. 
 
 By \autoref{sobolev-whole} there is a $c_2=c_2(d,\Theta)>0$ such that
 \begin{align*}
  \|v\tau\|_{L^{\Theta}(\R^d)}^2 & \leq  c_2 \Bigg(\int_{M_{\lambda 
r}}\int_{M_{\lambda r}} (v(x)\tau(x)-v(x)\tau(y))^2 \, \ma(x,\d y)\, \d x \\
  & \qquad \qquad + 2\int_{M_{\lambda r}}\int_{(M_{\lambda r})^c} 
(v(x)\tau(x)-v(x)\tau(y))^2 \, \ma(x,\d y)\, \d x \Bigg) \\
  & =: c_2 (I_1 + 2I_2).
 \end{align*}
We have
\begin{align*}
 I_1 & \leq \frac14 \Bigg(\int_{M_{\lambda r}}\int_{M_{\lambda r}} 
2[(v(y)-v(x))(\tau(x)+\tau(y))]^2 \, \ma(x,\d y)\, \d x\\
 & \qquad \quad + \int_{M_{\lambda r}}\int_{M_{\lambda 
r}}2[(v(x)+v(y))(\tau(x)-\tau(y))]^2 \, \ma(x,\d y)\, \d x\Bigg) \\
 & = \frac12 (J_1+J_2),
\end{align*}
Using $(\tau(x)+\tau(y))\leq 2$ for all $x,y\in M_{\lambda r}$ leads to
\[
 J_1 \leq  4\int_{M_{\lambda r}}\int_{M_{\lambda r}} (u(y)-u(x))^2\, \ma(x,\d 
y)\,\d x. \]
By $(v(x)+v(y))^2(\tau(x)-\tau(x))^2 \leq 2v(x)^2(\tau(x)-\tau(x))^2 + 
2v(y)^2(\tau(x)-\tau(x))^2$ and \autoref{lem:cut-off}, we have

\[ J_2 \leq 4\|v\|^2_{L^2(M_{\lambda r})} \sup_{x\in\R^d} \int_{\R^d} 
(\tau(y)-\tau(x))^2\ma(x,\d y) 
 \leq c_3 r^{-2}\left(\sum_{k=1}^d 
(\lambda^{\alam/\alpha_k}-1)^{-\alpha_k}\right)\|u\|^2_{L^2(M_{\lambda 
r})}. \]
Moreover, by \autoref{quadrat}
\[ I_2  \leq c_4 r^{-\amax}\left(\sum_{k=1}^d 
(\lambda^{\amax/\alpha_k}-1)^{-\alpha_k}\right)\|u\|_{L^2(M_{\lambda r})}^2.  \]
Hence there is a constant $c_1$, independent of $x_0, \lambda, r, 
\alpha_1,\dots,\alpha_d$ and $u$, such that
\begin{align*}
 & \|u\|_{L^{\Theta}(M_r)}^2 = \|v\|_{L^{\Theta}(M_r)}^2 
=\|v\tau\|_{L^{\Theta}(M_r)}^2 \leq \|v\tau\|_{L^{\Theta}(\R^d)}^2 \\
 & \leq c_1 \left(\int_{M_{\lambda r}}\int_{M_{\lambda r}} (u(x)-u(y))^2 \, 
\ma(x,\d y)\, \, \d x+r^{-\amax}\left(\sum_{k=1}^d 
(\lambda^{\amax/\alpha_k}-1)^{-\alpha_k}\right)\|u\|^2_{L^2(M_{\lambda 
r})}\right). 
\end{align*}
\end{proof}

We deduce the following corollary.

\begin{corollary}\label{cor:sobolev-local}
 Let $x_0\in M_1$ and $r\in(0,1)$. Let $\lambda\in(1,r^{-1}]$ and $u\in 
V^{\mu}(M_{\lambda r}(x_0)\big|\R^d)$. 
 Let $\Theta=2\beta/(\beta-1)$. Then there is a $c_1>0$, independent of $x_0, 
\lambda, r, \alpha_1,\dots,\alpha_d$ and $u$, but depending on $d,\Theta$, such 
that
 \begin{align*}
   \|u\|_{L^{\frac{2\beta}{\beta-1}}(M_r(x_0))}^2 & \leq c_1 
\Bigg(\int_{M_{\lambda r}(x_0)}\int_{M_{\lambda r}(x_0)} (u(x)-u(y))^2 \, 
\mu(x,\d y)\, \d x\\
  & \qquad \quad  +r^{-\amax}\left(\sum_{k=1}^d 
(\lambda^{\amax/\alpha_k}-1)^{-\alpha_k}\right)\|u\|^2_{L^2(M_{\lambda 
r}(x_0))}\Bigg).  
 \end{align*}
\end{corollary}
\begin{proof}
Since by assumption $\rho:=\lambda r\leq 1$, the assertion follows immediately 
by \autoref{sobolev-domainaxes} and \autoref{assum:comparability}.
\end{proof}

\subsection{Poincar\'{e} inequality}

Finally, we establish a Poincar\'{e} inequality in our setting. Let 
$\Omega\subset\R^d$ be an open and bounded set. For $f\in L^1(\Omega)$, set
\[ [f]_{\Omega} := \dashint_{\Omega} f(x) \, \d x = 
\frac{1}{|\Omega|}\int_{\Omega} f(x) \, \d x. \]
\begin{lemma}\label{poincare1}
 Let $r\in(0,1]$ and $x_0\in M_1$. Assume $v\in V^{\mu}(M_r(x_0)\big|\R^d)$. 
There exists a constant $c_1>0$, 
 independent of $x_0, r$ and $v$, such that 
 \[ \|v-[v]_{M_r(x_0)}\|_{L^2(M_r(x_0))}^2 \leq 
c_1r^{\alam}\mathcal{E}^{\mu}_{M_r(x_0)}(v,v). \]
\end{lemma}
\begin{proof}
To simplify notation, we assume $x_0=0$. Via translation, the 
assertion follows  
for general $x_0\in\R^d$. Let 
\[ \gamma=\max\left\{ (\alpha_k(2-\alpha_k))^{-1}\colon k\in\{1,\dots,d 
\}\right\}. \]
 By Jensen's inequality
 \begin{align*}
  \|v-[v]_{M_r}\|_{L^2(M_r)}^2 & = \int_{M_r} \left(\frac{1}{|M_r|}\int_{M_r} 
(v(x)-v(y))\, \,\d y \right)^2\, \,\d x \\
  & \leq \frac{1}{|M_r|}\int_{M_r} \int_{M_r} (v(x)-v(y))^2\, \,\d y \, \,\d x 
:= J.
 \end{align*}
Let $\ell=(\ell_0(x,y),\dots,\ell_d(x,y))$ 
be a polygonal chain
connecting 
$x$ and $y$ with
\[ \ell_k(x,y) = (l^k_1,\dots,l^k_d), \quad \text{where } \begin{cases}
                                                       l^k_j = y_j, &\text{if } 
j\leq k, \\
                                                       l^k_j = x_j, &\text{if } 
j>k.
                                                      \end{cases} \]
 Then                                                     
\begin{align*}
 J  \leq \frac{d}{|M_r|}\sum_{k=1}^d \int_{M_r} \int_{M_r} 
(v(\ell_{k-1}(x,y))-v(\ell_k(x,y)))^2\, \d y \, \d x  := 
\frac{d}{|M_r|}\sum_{k=1}^d I_k.
 \end{align*}
We fix $k\in\{1,\dots,d\}$ and set $w=\ell_{k-1}(x,y) = 
(y_1,\dots,y_{k-1},x_k,\dots,x_d)$. Let $z:= 
x+y-w=(x_1,\dots,x_{k-1},y_k,\dots,y_d)$. Then
$\ell_k(x,y) = w+e_k(z_k-w_k) = (y_1,\dots,y_{k},x_{k+1},\dots,x_d)$. By 
Fubini's Theorem
\begin{align*}
 I_k & = \left(\prod_{j\neq 
k}\int_{-r^{\alam/\alpha_j}}^{r^{\alam/\alpha_j}}\,\d 
z_j\right)\int_{-r^{\alam/\alpha_1}}^{r^{\alam/\alpha_1}} 
\cdots \int_{-r^{\alam/\alpha_d}}^{r^{\alam/\alpha_d}} 
\int_{-r^{\alam/\alpha_k}}^{r^{\alam/\alpha_k}} 
(v(w)-v(w+e_k(z_k-w_k)))^2\,\\
 & \hspace*{9cm} \,\d z_k\, \d w_d \cdots \,\d w_1  \\
 & \leq 2^{d-1}r^{\sum_{j\neq k} \alam/\alpha_j} \int_{M_r} 
\int_{-r^{\alam/\alpha_k}}^{r^{\alam/\alpha_k}} 
(v(w)-v(w+e_k(z_k-w_k)))^2 
\frac{2^{1+\alpha_k}r^{\alam/\alpha_k+2}}{|x_k-z_k|^{1+\alpha_k}}\,\d 
z_k\, \d w \\ 
 & \leq 42^{d}r^{\alam\beta}r^{\alam} \int_{M_r} 
\int_{-r^{\alam/\alpha_k}}^{r^{\alam/\alpha_k}} 
(v(w)-v(w+e_k(z_k-w_k)))^2 \frac{1}{|x_k-z_k|^{1+\alpha_k}}\,\d z_k\, \d w \\ 
 & \leq \frac{4r^{\alam}}{\alpha_k(2-\alpha_k)}|M_r| \int_{M_r} 
\int_{-r^{\alam/\alpha_k}}^{r^{\alam/\alpha_k}} 
(v(w)-v(w+e_k(z_k-w_k)))^2 
\frac{\alpha_k(2-\alpha_k)}{|x_k-z_k|^{1+\alpha_k}}\,\d z_k\, \d w\\
 & \leq 4r^{\alam}\gamma|M_r| \int_{M_r} 
\int_{-r^{\alam/\alpha_k}}^{r^{\alam/\alpha_k}} 
(v(w)-v(w+e_k(z_k-w_k)))^2 
\frac{\alpha_k(2-\alpha_k)}{|x_k-z_k|^{1+\alpha_k}}\,\d z_k\, \d w.
\end{align*}
Hence there are $c_1,c_2>0$, independent of $\rho$, $v$ and $x_0$, but 
depending on $d$ and $\gamma$, such that
\begin{align*}
 &\|v -[v]_{M_r}\|_{L^2(M_r)}^2\\
 &\leq c_2r^{\alam}\sum_{k=1}^d \int_{M_r} 
\int_{-r^{\alam/\alpha_k}}^{r^{\alam/\alpha_k}} 
(v(w)-v(w+e_k(z_k-w_k)))^2 
\frac{\alpha_k(2-\alpha_k)}{|x_k-z_k|^{1+\alpha_k}}\,\d z_k\, \d w \\
 & = c_2r^{\alam} \mathcal{E}^{\ma}_{M_r}(v,v)\leq 
c_1r^{\alam}\mathcal{E}^{\mu}_{M_r}(v,v),
\end{align*}
where we used \autoref{assum:comparability} in the last inequality.
\end{proof}

\section{Properties of weak supersolutions}\label{sec:prop_weak_sol} 
\allowdisplaybreaks

In this section we prove our main auxiliary result, that is a weak Harnack 
inequality for weak
supersolutions
using the Moser iteration technique.
For this purpose we establish a Poincaré inequality and show that the 
logarithm of weak supersolutions are functions 
of bounded mean oscillation.

Let $\lambda>0$, $\Omega\subset\R^d$ be open, $u\in V^{\ma}(\Omega|\R^d)$ and 
$\Psi:\R^d\to\R^d$ be a diffeomorphism defined by
\[ \Psi(x)=\begin{pmatrix}
         \lambda^{\frac{\alam}{\alpha_1}}  & \cdots & 0 \\
         \vdots &  \ddots &  0 \\
         0  & 0 & \lambda^{\frac{\alam}{\alpha_d}}
        \end{pmatrix}x.\]
Then by change of variables, the energy form $\mathcal{E}^{\ma}_{\Omega}$ 
behaves as follows
 \[  \mathcal{E}_{\Omega}^{\ma}(u\circ\Psi,u\circ\Psi)= 
\lambda^{\alam-\alam\beta}\mathcal{E}^{\ma}_{\Psi(\Omega)}(u,
u). \]

The next lemma provides a key estimate for $\log u$.

\begin{lemma}\label{lem:morrey-log}
Let $x_0\in M_1$, $r\in(0,1]$ and $\lambda>1$. Assume $f\in L^{q}(M_{\lambda 
r}(x_0))$ 
 for some $q>2$. Assume $u\in V^{\mu}(M_{\lambda r}(x_0)\big|\R^d)$ is 
nonnegative in $\R^d$ and satisfies
 \begin{equation}\label{supersolution}
   \begin{aligned}
  \mathcal{E}^{\mu}(u,\varphi)&\geq (f,\varphi) \quad \text{for any nonnegative 
} 
\varphi\in H^{\mu}_{M_{\lambda r}(x_0)}(\R^d),\\
  u(x)&\geq \epsilon \qquad \ \text{ for almost all } x\in M_{\lambda r}(x_0) 
\text{ and some } \epsilon>0.
 \end{aligned}
 \end{equation}
There exists a constant $c_1>0$, independent of $x_0, \lambda, r, 
\alpha_1,\dots,\alpha_d$ and $u$, such that
  \begin{align*}
  \int_{M_r(x_0)}&\int_{M_r(x_0)} \left( \sum_{k=1}^\infty \frac{(\log u(y) - 
\log u(x))^{2k}}{(2k)!} \right)\mu(x,\d y)\,\d x \\
  & \leq c_1\left(\sum_{k=1}^d 
\left(\lambda^{\alam/\alpha_k}-1\right)^{-\alpha_k}\right)r^{-\alpha_{
\max}}|M_{\lambda r}(x_0)|
  +\epsilon^{-1}\|f\|_{L^{q}(M_{\lambda r}(x_0))}|M_{\lambda 
r}(x_0)|^{\frac{q}{q-1}}.
 \end{align*}
\end{lemma}
\begin{proof}
We follow the lines of \cite[Lemma 4.4]{DyKa15}.

Let $\tau:\R^d\to\R$ be as in \eqref{def:cutoff_aniso}.
Then by \autoref{lem:cut-off} and \eqref{eq:assum:cutoff}, there is $c_2>0$, 
such that

\[ \sup_{x\in\R^d} \int_{\R^d} (\tau(y)-\tau(x))^2\mu(x,\d y)\leq 
c_2r^{-\alam}\left(\sum_{k=1}^d 
(\lambda^{\alam/\alpha_k}-1)^{-\alpha_k}\right). \]

For brevity, we write $M_{\lambda r}(x_0)=M_{\lambda r}$ and $M_{r}(x_0)=M_r$ 
within this proof. By definition of $\tau$ and \eqref{eq:assum:symmetry},
\begin{equation}\label{eq:gradtauest}
\begin{aligned}
\int_{\R^d}\int_{\R^d} &  \left(\tau(y) - \tau(x) \right)^2 \mu(x,\d y)\,\d x \\
& = \int_{M_{\lambda r}}\int_{M_{\lambda r}}  \left(\tau(y) - \tau(x) \right)^2 
\mu(x,\d y)\,\d x + 2 \int_{M_{\lambda  r}}\int_{M_{\lambda  r}^c} 
\left(\tau(y) - \tau(x) \right)^2 \mu(x,\d y)\,\d x \\
& \leq 2 \int_{M_{\lambda r}}\int_{\R^d} \left(\tau(y) - \tau(x) \right)^2 
\mu(x,\d y)\,\d x \\
& \leq  2 |M_{\lambda  r}| \sup\limits_{x \in \R^d} \int_{\R^d} \left(\tau(y) - 
\tau(x)\right)^2 \mu(x,\d y)\\
& \leq  c_3\left(\sum_{k=1}^d 
(\lambda^{\alam/\alpha_k}-1)^{-\alpha_k}\right)r^{-\alam} 
|M_{\lambda r}|.
\end{aligned} 
\end{equation}
Let $-\varphi(x) = - \tau^2(x) u^{-1}(x)\leq 0$. By \eqref{supersolution}, we 
deduce as in the proof 
of \cite[Lemma 3.3]{Kas09}
 \begin{align*}
(f,-\varphi) &\geq  \int_{\R^d} \int_{\R^d} \left(u(y) - u(x)\right) 
\left(\tau^2(x) u^{-1}(x) - \tau^2(y) u^{-1}(y)\right)  \mu(x,\d y)\, \d x \\
&= \int_{M_{\lambda r}}\int_{M_{\lambda r}} \tau(x)\tau(y) \left( 
\frac{\tau(x)u(y)}{\tau(y)u(x)} + \frac{\tau(y) u(x)}{\tau(x) u(y)} - 
\frac{\tau(y)}{\tau(x)}-\frac{\tau(x)}{\tau(y)} \right) \, \mu(x, \d y)\, \d x 
\\
&\quad + 2 \int_{M_{\lambda r}}\int_{M_{\lambda r}^c} \left(u(y) - u(x)\right) 
\left(\tau^2(x) u^{-1}(x) -\tau^2(y) u^{-1}(y)\right) \; \mu(x, \d y) \,\d x \\
&\geq \int_{M_r} \int_{M_r} \left( 2 \sum^\infty_{k=1} \frac{\left( 
\log u(y) - \log u(x) \right)^{2k}}{(2k)!} \right) \, \mu(x,\d y)\,\d  x \\
& \quad -  \int_{M_{\lambda r}}\int_{M_{\lambda r}} \left( \tau(x) - \tau(y) 
\right)^2 \, \mu(x,\d y)\,\d x \\
&\quad + 2 \int_{M_{\lambda r}}\int_{M_{\lambda r}^c} \left(u(y) - u(x)\right) 
\left(\tau^2(x) u^{-1}(x) - \tau^2(y) u^{-1}(y)\right) \; \mu(x,\d y)\,\d x,
\end{align*}
Using the nonnegativity of $u$ in $\R^d$, the third term on the right-hand 
side can be estimated as follows:
\begin{align*}
2 &\int_{M_{\lambda r}}\int_{M_{\lambda r}^c} \left(u(y) - u(x)\right) 
\left(\tau^2(x) u^{-1}(x)- \tau^2(y) u^{-1}(y)\right) \; \mu(x,\d y)\,\d x \\
&= 2 \int_{M_{\lambda r}}\int_{M_{\lambda r}^c} \left(u(y) - u(x)\right) 
\left(- \tau^2(x)u^{-1}(x)\right) \; \mu(x,\d y)\,\d x \\
&= 2 \int_{M_{\lambda r}} \int_{M_{\lambda r}^c} \frac{\tau^2(x)}{u(x)}  u(y) 
\, \mu(x,\d y)\,\d x - 2 \int_{M_{\lambda r}} \int_{M_{\lambda r}^c}  \tau^2(x) 
\, \mu(x,\d y) \, \,\d x \\ 
&\geq - 2 \int_{\R^d}\int_{\R^d} \left(\tau(y)-\tau(x)\right)^2 \mu(x,\d y)\,\d 
x,
\end{align*}
Therefore, by the H\"older 
inequality and $|u^{-1}| \leq\epsilon^{-1}$
\begin{equation}\label{eq:split3}
\begin{aligned}
\int_{M_r} &\int_{M_r} \left( 2 \sum^\infty_{k=1} \frac{ \left( \log u(y) - 
\log u(x) \right)^{2k}}{(2k)!} \right) \, \mu(x,\d y)\,\d x \\
&\leq 3 \int_{M_{\lambda r}}\int_{\R^d} \left( \tau(x) - \tau(y) \right)^2 \, 
\mu(x,\d y)\,\d x + (f,-\tau^2u^{-1}) \\
& \leq c_1\left(\sum_{k=1}^d 
\left(\lambda^{\alam/\alpha_k}-1\right)^{-\alpha_k}\right)r^{-2} 
|M_{\lambda r}| + \|f\|_{L^{q}(M_{\lambda r})} 
\|u^{-1}\|_{L^{q/(q-1)}(M_{\lambda r})} \\
& \leq c_1\left(\sum_{k=1}^d 
\left(\lambda^{\alam/\alpha_k}-1\right)^{-\alpha_k}\right)r^{-2} 
|M_{\lambda r}| + \epsilon^{-1}\|f\|_{L^{q}(M_{\lambda r})} |M_{\lambda 
r}|^{q/(q-1)}.
\end{aligned} 
\end{equation}
\end{proof}
It is a direct consequence of \autoref{poincare1} for $v=\log(u)$ 
and \autoref{lem:morrey-log}. 
\begin{corollary}\label{L2-estimate}
Let $x_0\in M_1$, $r\in(0,1]$ and $\lambda\in[\frac54,2]$. Let $f\in 
L^{q}(M_{2r}(x_0))$ 
 for some $q>2$. Assume $u\in V^{\mu}(M_{\lambda r}(x_0)\big|\R^d)$ is 
nonnegative in $\R^d$ and satisfies
 \begin{equation}
   \begin{aligned}
  \mathcal{E}^{\mu}(u,\varphi)&\geq (f,\varphi) \quad \text{for any nonnegative 
} 
\varphi\in H^{\mu}_{M_{\lambda r}(x_0)}(\R^d),\\
  u(x)&\geq \epsilon \qquad \ \text{ for almost all } x\in M_{2r} \text{ and } 
\epsilon>r^{\alam}\|f\|_{L^{q}(M_{\lambda r}(x_0))}.
 \end{aligned}
 \end{equation}
Then there exists a constant $c_1>0$, independent of $x_0, r$ and $u$, 
such that 
\begin{equation}
 \|\log u-[\log u]_{M_r(x_0)}\|_{L^2(M_r(x_0))}^2 \leq c_2|M_r(x_0)|.
\end{equation}
\end{corollary}
\begin{proof}
Set $M_r=M_r(x_0)$ and $M_{\lambda r}=M_{\lambda r}(x_0)$. Note
\begin{equation}\label{doubling}
  \begin{aligned}
  |M_{\lambda r}|& = \left(\prod_{k=1}^d 2 (\lambda 
r)^{\alam/\alpha_k}\right) = \lambda^{\alam\beta}2^d 
r^{\alam\beta} = \lambda^{\alam\beta} |M_r|\leq 
2^{2\beta}|M_r|, \\
   |M_{\lambda r}|^{\frac{q}{q-1}} &= \lambda^{\alam\beta\frac{q}{q-1}} 
2^{d\frac{q}{q-1}}r^{\alam\beta\frac{q}{q-1}} \leq \lambda^{4\beta} 
2^{2d}r^{\alam\beta} = 2^{4\beta+d} |M_r|, 
  \end{aligned}
\end{equation}
where we used the facts $\lambda\leq 2$, 
$r\leq 1$ and $\max\{x/(x-1)\colon x\geq 2\}=2$.\\
By \autoref{poincare1} for $v:=\log(u)$, \autoref{lem:morrey-log} and 
\eqref{doubling}, we observe
  \begin{align*}
   \|&\log u-[\log u]_{M_r}\|_{L^2(M_r)}^2 \leq 
c_1r^{\alam}\mathcal{E}_{M_r(x_0)}(\log u,\log u) \\
   & \leq 2c_1r^{\alam}\int_{M_r}\int_{M_r} \left( \sum_{k=1}^\infty 
\frac{(\log u(y) - \log u(x))^{2k}}{(2k)!} \right)\mu(x,\d y)\,\d x \\
   & \leq 2c_1r^{\alam}\left(c_3\left(\sum_{k=1}^d 
\left(\lambda^{\frac{\alam}{\alpha_k}}-1\right)^{-\alpha_k}\right)r^{
-\alam} |M_{\lambda r}| + \epsilon^{-1}\|f\|_{L^{q}(M_{\lambda r})} 
|M_{\lambda r}|^{q/(q-1)}\right) \\
   & \leq 2c_1r^2\left(c_3\left(\sum_{k=1}^d 
\left(\left(\frac{5}{4}\right)^{\frac{\alam}{\alpha_k}}-1\right)^{
-\alpha_k}\right)r^{-2} |M_{\lambda r}| + r^{-\alam} |M_{\lambda 
r}|^{q/(q-1)}\right) \\
   & \leq 2c_1\left(c_3 c_4d |M_{\lambda r}| + 2^{4\beta+d}|M_r|\right) \\
   & = 2c_1\left(c_3 c_4d 2^{2\beta}|M_{r}| + 2^{4\beta+d}|M_r|\right) = 
c_1(d,\beta)|M_r|.
  \end{align*}
 Here we have used the fact, that there is a $c_4=c_4(\alam)>0$ such that 
$\max\{(5/4)^{{\alam}/x}-1)^{-x} \colon 
x\in(0,\alam]\}\leq c_4$.
\end{proof}

A consequence of the foregoing results is the following theorem.

\begin{theorem}\label{flip-assert}
 Assume $x_0\in M_1$, $r\in(0,1]$ and $f\in L^{q}(M_{\frac{5}{4}r}(x_0))$ for 
some $q>2$. 
 Assume $u\in V^{\mu}(M_{\frac{5}{4}r}(x_0)\big|\R^d)$ is nonnegative in $\R^d$ 
and satisfies
 \begin{align*}
  \mathcal{E}^{\mu}(u,\varphi)&\geq (f,\varphi) \quad \text{for any nonnegative 
} 
\varphi\in H^{\mu}_{M_{\frac{5}{4}r}(x_0)}(\R^d),\\
  u(x)&\geq \epsilon \qquad \ \text{ for almost all } x\in M_{\frac{5}{4}r} 
\text{ and some } 
\epsilon>r^{\alam}\|f\|_{L^{q}(M_{\frac{5}{4}r}(x_0))}.
 \end{align*}
 Then there exist $\overline{p}\in(0,1)$ and $c_1>0$, independent of $x_0, r, 
u$ and $\epsilon$, such that
\begin{equation}
 \left(\dashint_{M_r(x_0)} u(x)^{\overline{p}} \; \, \d x 
\right)^{1/\overline{p}} \; \d x \leq c_1 \left(\dashint_{M_r(x_0)} 
u(x)^{-\overline{p}} \; \d x \right)^{-1/\overline{p}}.
\end{equation}
\end{theorem}
\begin{proof}  
This proof follows the proof of \cite[Lemma 4.5]{DyKa15}.

The main idea is to prove $\log u \in$ BMO$(M_r(x_0))$ and use the 
John-Nirenberg inequality for doubling metric measure spaces. Let $x_0\in M_1$ 
and $r\in(0,1]$.
Endowed with the Lebesgue measure, the metric measure space
$(M_r(x_0),d,\d x)$ is a doubling space. 
Let $z_0 \in M_r(x_0)$ and $\rho > 0$ such that $M_{2\rho}(z_0) \subset 
M_{r}(x_0)$. 
Note that by \eqref{doubling} $|M_{2\rho}|^{\frac{q}{q-1}} \leq 
2^{4\beta+d}|M_{\rho}|$. 
\autoref{L2-estimate} and the H\"older inequality imply
\begin{align*}
\int_{M_{\rho}(z_0)} \Big|\log u(x) - [\log u]_{M_{\rho}(z_0)}  \Big| \; \d x& 
\leq \|\log u - [\log u]_{M_{\rho}(z_0)} \|_{L^2(M_{\rho}(z_0))} 
\sqrt{|M_{\rho}|} \\
& \leq c_2|M_{\rho}|.
\end{align*}
This proves $\log u \in$ BMO$(M_r(x_0))$. The John-Nirenberg inequality 
\cite[Theorem 19.5]{HKM06} states, that $\log u \in$ BMO$(M_r(x_0))$, iff for 
each  $M_{\rho}\Subset 
M_r(x_0)$ and $\kappa >0$
\begin{equation}\label{John-Nierenberg}
 |\{ x\in M_{\rho} \colon |\log u(x)-[\log u]_{M_{\rho}}|>\kappa \}|\leq 
c_3e^{-c_4 \kappa}|M_{\rho}|, 
\end{equation}
where the positive constants $c_3,c_4$ and the BMO norm depend only on each 
other, the dimension $d$ and the doubling constant.\\
By Cavalieri's principle, we have for $h:M_R(x_0)\to[0,\infty]$, using the 
change of variable $t=e^{\kappa}$, that
\begin{align*}
 \dashint_{M_r(x_0)} e^{h(x)} \, \, \d x &  = \frac{1}{|M_r|}\Bigg(\int_0^1 
|\{x\in M_r(x_0) \colon e^{h(x)}>t\}|\; \d t 
\\
 &  \hspace*{3cm}+\int_1^\infty |\{x\in M_r(x_0) \colon e^{h(x)}>t\}|\, \d 
t\Bigg) \\	
 & \leq 1+\frac{1}{|M_r|}\int_0^\infty e^{\kappa}|\{x\in M_r(x_0) \colon 
h(x)>\kappa\}|\; \d \kappa.
\end{align*}
Let $\overline{p}\in(0,1)$ be chosen such that $\overline{p}<c_4$. The 
application of \eqref{John-Nierenberg} implies
\begin{align*}
 \dashint_{M_r(x_0)} & \exp\left(\overline{p} |\log u(y) - [\log 
u]_{M_r(x_0)}|\right) \; \d y \\
 & \leq 1+\int_0^\infty e^{\kappa}\frac{|\{x\in M_r(x_0) \colon |\log u(x) - 
[\log u]_{M_r(x_0)}|>\kappa/\overline{p}\}|}{|M_r|}\; \d \kappa \\
 & \leq 1+\int_0^\infty e^{\kappa}\frac{c_3 e^{-c_4 \kappa/\overline{p}} 
|M_r(x_0)|}{|M_r|}\; \d \kappa \\
 & \leq 1+c_3\int_0^\infty e^{(1-c_4/\overline{p})\kappa}\; \d \kappa \\
 & = 1+\frac{c_3}{c_4/\overline{p}-1} = \frac{c_4-\overline{p} + 
c_3\overline{p}}{c_4-\overline{p}}=:c_5<\infty.
\end{align*}
Hence
\begin{align*}
 & \left(\dashint_{M_r(x_0)}u(y)^{\overline{p}} \; \d y\right) 
\left(\dashint_{M_r(x_0)}u(y)^{-\overline{p}} \; \d y\right) \\
 &= \left(\dashint_{M_r(x_0)}e^{\overline{p}(\log u(y) - [\log u]_{M_r})} \; \d 
y\right) \left(\dashint_{M_r(x_0)}e^{-\overline{p}(\log u(y) - [\log u]_{M_r})} 
\; \d y\right) \leq c_5^2=c_1.
\end{align*}
\end{proof}

\subsection{The weak Harnack inequality}\label{subsec:weakharnack} 
\allowdisplaybreaks

In this subsection we prove the weak Harnack inequality 
\autoref{theo:weakharnack}, using
the Moser iteration technique for negative exponents.

 Let $\Omega\subset\R^d$ be open and bounded. Let $q>\beta$ and $f\in 
L^\frac{\beta}{\beta-1}(\Omega)$. Then Lyapunov's inequality implies for any 
$a>0$
\begin{equation}\label{inequal}
 \left\|f\right\|_{L^{\frac{q}{q-1}}(\Omega)} \leq 
\frac{\beta}{q}a\left\|f\right\|_{L^{\frac{\beta}{\beta-1}}(\Omega)} + 
\frac{q-\beta}{q}a^{-\beta/(q-\beta)}\left\|f\right\|_{L^1(\Omega)}.
\end{equation}
\begin{lemma}\label{prop:ab}
 There exist positive constants $c_1,c_2>0$ such that for every $a,b>0$, $p>1$ 
and $0\leq \tau_1,\tau_2\leq 1$ the following is true:
 \[(b-a)(\tau_1^2a^{-p}-\tau_2^2b^{-p})\geq 
 c_1 \left(\tau_1a^{\frac{-p+1}{2}}-\tau_2b^{\frac{-p+1}{2}}\right)^2 
- \frac{c_2 p}{p-1}(\tau_1-\tau_2)^2(b^{-p+1}+a^{-p+1}) \,.\]
\end{lemma}
A proof of \autoref{prop:ab} can be found in the published version of 
\cite{DyKa15}.

\begin{lemma}\label{abschfuermoser}
 Assume $x_0\in M_1$ and $r\in[0,1)$. Moreover, let 
$\lambda\in(1,\min\{r^{-1},\sqrt{2}\})$ and $f\in L^{q}(M_{\lambda r}(x_0))$ 
for some $q>\max\{2,\beta\}$. 
 Assume $u\in V^{\mu}(M_{\lambda r}(x_0)\big|\R^d)$ satisfies
 \begin{align*}
  \mathcal{E}(u,\varphi)&\geq (f,\varphi) \quad \text{for any nonnegative } 
\varphi\in 
H^{\mu}_{M_{\lambda r}(x_0)}(\R^d),\\
  u(x)&\geq \epsilon \qquad \ \text{ for a.a. } x\in M_{\lambda r}(x_0) \text{ 
and some } \epsilon>\|f\|_{L^q(M_{\lambda r}(x_0))}r^{\alam(q-\beta)/q}.
 \end{align*}
Then for any $p>1$, there is a $c_1>0$ independent of $u, x_0, r, p, 
\alpha_1,\dots,\alpha_d$ and $\epsilon$, such that
\begin{align*}
  &\left\| u^{-1} \right\|_{L^{(p-1)\frac{\beta}{\beta-1}}(M_r(x_0))}^{p-1} 
\leq c_1\frac{p}{p-1}\left(\sum_{k=1}^d 
(\lambda^{\alam/\alpha_k}-1)^{-\alpha_k}\right)r^{-\alam}\left\|u^{-1}\right\|_{
L^{p-1}
(M_{\lambda r}(x_0))}^{p-1}.
\end{align*}
 \end{lemma}
\begin{proof}
 Let $\tau:\R^d\to\R$ be as in \eqref{def:cutoff_aniso}. We follow the idea 
of the proof of \cite[Lemma 4.6]{DyKa15}.
  
For brevity let $M_r=M_r(x_0)$. Since $\mathcal{E}(u,\varphi)\geq 
(f,\varphi)$ 
for any nonnegative $\varphi\in H_{M_r}(\R^d)$ we get
\[ \mathcal{E}(u,-\tau^{2}u^{-p})\leq (f,-\tau^{2}u^{-p}). \]
Furthermore
\begin{align*}
  \int_{\R^d}\int_{\R^d}& \left(u(y)-u(x)\right)\left(\tau(x)^2u(x)^{-p} - 
\tau(y)^2u(y)^{-p}\right) \mu(x,\d y)\, \d x  \\
& = \int_{M_{\lambda r}}\int_{M_{\lambda r}} 
\left(u(y)-u(x)\right)\left(\tau(x)^2u(x)^{-p} - \tau(y)^2u(y)^{-p}\right) 
\mu(x,\d y)\, \d x \\
& \qquad \quad +   2\int_{M_{\lambda r}}\int_{M_{\lambda r}^c} 
\left(u(y)-u(x)\right)\left(\tau(x)^2u(x)^{-p} - \tau(y)^2u(y)^{-p}\right) 
\mu(x,\d y)\, \d x  \\
& =: J_1 + 2J_2.
\end{align*}

We first study $J_2$. By \autoref{lem:cut-off} and \eqref{eq:assum:cutoff},
\begin{align*}
J_2 & = \int_{M_{\lambda r}}\int_{M_{\lambda r}^c} u(y)\tau(x)^2u(x)^{-p} 
\mu(x,\d y)\, \d x  -  \int_{M_{\lambda r}}\int_{M_{\lambda r}^c} 
\tau(x)^2u(x)^{-p+1} \mu(x,\d y)\, \d x \\
& \geq -\int_{M_{\lambda r}}\int_{M_{\lambda r}^c} \tau(x)^2u(x)^{-p+1} 
\mu(x,\d y)\, \d x\\
& \geq - \|u^{-p+1}\|_{L^1(M_{\lambda r})}\sup_{x\in\R^d} \int_{\R^d} 
(\tau(y)-\tau(x))^2\mu(x,\d y) \\
& \geq - \|u^{-1}\|_{L^{p-1}(M_{\lambda r})}^{p-1} 
c_2r^{-\alam}\left(\sum_{k=1}^d 
\left(\lambda^{\alam/\alpha_k}-1\right)^{-\alpha_k}\right).
 \end{align*}
 \allowdisplaybreaks
Applying \autoref{prop:ab} for $a=u(x), b=u(y), 
\tau_1=\tau(x),\tau_2=\tau(y)$ on $J_1$, 
there exist $c_3,c_4>0$ such that
\begin{align*}\allowdisplaybreaks
J_1& = \int_{M_{\lambda r}}\int_{M_{\lambda r}} 
\left(u(y)-u(x)\right)\left(\tau(x)^2u(x)^{-p} - \tau(y)^2u(y)^{-p}\right) 
\mu(x,\d y)\, \d x \\
& \geq c_3 \int_{M_{\lambda r}}\int_{M_{\lambda r}} 
\left(\tau(x)u(x)^{\frac{-p+1}{2}}-\tau(x)u(x)^{\frac{-p+1}{2}}\right)^2 
\mu(x,\d y)\, \d x \\
& \qquad \quad - c_4\frac{p}{p-1}\int_{M_{\lambda r}}\int_{M_{\lambda r}} 
(\tau(y)-\tau(x))^2(u(y)^{-p+1}+u(x)^{-p+1})\mu(x,\d y)\, \d x. 
 \end{align*}
 Hence 
 \begin{equation}\label{moserprep}\allowdisplaybreaks
\begin{aligned}
 \int_{M_{\lambda r}}&\int_{M_{\lambda r}} 
\left(\tau(x)u(x)^{\frac{-p+1}{2}}-\tau(x)u(x)^{\frac{-p+1}{2}}\right)^2 
\mu(x,\d y)\, \d x \\
  & \leq \frac{1}{c_3}J_1 + c_4\frac{p}{p-1}\int_{M_{\lambda 
r}}\int_{M_{\lambda r}} (\tau(y)-\tau(x))^2(u(y)^{-p+1}+u(x)^{-p+1})\mu(x,\d 
y)\, \d x \\
  & = \frac{1}{c_3}  \int_{\R^d}\int_{\R^d} 
\left(u(y)-u(x)\right)\left(\tau(x)^2u(x)^{-p} - \tau(y)^2u(y)^{-p}\right) 
\mu(x,\d y)\, \d x - \frac{2}{c_3}J_2 \\
  & \qquad + c_4\frac{p}{p-1}\int_{M_{\lambda r}}\int_{M_{\lambda r}} 
(\tau(y)-\tau(x))^2(u(y)^{-p+1}+u(x)^{-p+1})\mu(x,\d y)\, \d x \\
  & \leq \frac{1}{c_3}  \int_{\R^d}\int_{\R^d} 
\left(u(y)-u(x)\right)\left(\tau(x)^2u(x)^{-p} - \tau(y)^2u(y)^{-p}\right) 
\mu(x,\d y)\, \d x \\
  & \qquad + \frac{16}{c_3}\|u^{-p+1}\|_{L^1(M_{\lambda r})} 
r^{-\alam}\left(\sum_{k=1}^d 
\left(\lambda^{\alam/\alpha_k}-1\right)^{-\alpha_k}\right) 
\\
  & \qquad + c_4\frac{p}{p-1}\int_{M_{\lambda r}}\int_{M_{\lambda r}} 
(\tau(y)-\tau(x))^2(u(y)^{-p+1}+u(x)^{-p+1})\mu(x,\d y)\, \d x. 
\end{aligned} 
 \end{equation}
We derive the assertion from \eqref{moserprep}.

The first expression of the right-hand-side of \eqref{moserprep} can be 
estimated 
with the help of \eqref{inequal} as follows:
\begin{align*}
 \int_{\R^d}\int_{\R^d}& \left(u(y)-u(x)\right)\left(\tau(x)^2u(x)^{-p} - 
\tau(y)^2u(y)^{-p}\right) \mu(x,\d y)\, \d x = \mathcal{E}(u,-\tau^2u^{-p}) \\
 & \leq (f,-\tau^2u^{-p}) \leq \epsilon^{-1} |(f,-\tau^2u^{-p+1})| = 
\epsilon^{-1} |(\tau f,\tau u^{-p+1})| \\
 & \leq \epsilon^{-1} \|\tau f\|_{L^q(\R^d)}\|\tau 
u^{-p+1}\|_{L^{\frac{q}{q-1}}(\R^d)} \\
 & \leq \epsilon^{-1} \|\tau f\|_{L^q(\R^d)}\left(\frac{\beta}{q}a\|\tau 
u^{-p+1}\|_{L^{\frac{\beta}{\beta-1}}(\R^d)}+\frac{q-\beta}{q}a^{\frac{-\beta}{
q-\beta}}\|\tau u^{-p+1}\|_{L^1(\R^d)}\right) \\
 & \leq \epsilon^{-1} \|f\|_{L^q(M_{\lambda r})}\left(\frac{\beta}{q}a\|\tau 
u^{-p+1}\|_{L^{\frac{\beta}{\beta-1}}(\R^d)}+\frac{q-\beta}{q}a^{\frac{-\beta}{
q-\beta}}\|\tau u^{-p+1}\|_{L^1(\R^d)}\right) \\
 & \leq r^{\alam(\beta-q)/q}\left(\frac{\beta}{q}a\|\tau 
u^{-p+1}\|_{L^{\frac{\beta}{\beta-1}}(\R^d)}+\frac{q-\beta}{q}a^{\frac{-\beta}{
q-\beta}}\|\tau u^{-p+1}\|_{L^1(\R^d)}\right),
\end{align*}
where $a>0$ can be chosen arbitrarily.
Set 
\[a=r^{\alam(q-\beta)/q}\omega\] 
for some $\omega>0$. Since $1<\lambda\leq \sqrt{2}$,  for 
all $k\in\{1,\dots,d\}$
\[\lambda\leq (2^{1/\alpha_k}+1)^{\alpha_k/2} \iff 
(\lambda^{2/\alpha_k}-1)^{-\alpha_k}\geq \frac12.\]
Using $(\lambda^{\alam/\alpha_k}-1)^{-\alpha_k} \geq 
(\lambda^{2/\alpha_k}-1)^{-\alpha_k}$, leads to
\[ \left(\sum_{k=1}^d (\lambda^{2/\alpha_k}-1)^{-\alpha_k}\right)\geq 1. \]  
Altogether, we obtain
\begin{equation}\label{rhs1}
\begin{aligned}
  \int_{\R^d}\int_{\R^d}& \left(u(y)-u(x)\right)\left(\tau(x)^2u(x)^{-p} - 
\tau(y)^2u(y)^{-p}\right) \mu(x,\d y)\, \d x \\
  & \leq 
\frac{\beta}{q}\omega\|u^{-p+1}\|_{L^{\frac{\beta}{\beta-1}}(M_{\lambda 
r})}+\frac{q-\beta}{\beta}r^{-\alam}\omega^{\frac{-\beta}{q-\beta}}\|\tau 
u^{-p+1}\|_{L^1(M_{\lambda r})} \\
  & \leq 
\frac{\beta}{q}\omega\|u^{-1}\|_{L^{(p-1)\frac{\beta}{\beta-1}}(M_{\lambda 
r})}^{p-1}\\
  & \qquad +\frac{q-\beta}{\beta}r^{-\alam}\left(\sum_{k=1}^d 
(\lambda^{\alam/\alpha_k}-1)^{-\alpha_k}\right)\omega^{\frac{-\beta}{q-\beta}}
\|u^{
-1}\|_{L^{p-1}(M_{\lambda r})}^{p-1}.
\end{aligned}
\end{equation}
The third expression of the right-hand-side of \eqref{moserprep} can be 
estimated as 
follows:
\begin{align*}
  \int_{M_{\lambda r}}\int_{M_{\lambda r}} 
&(\tau(y)-\tau(x))^2(u(y)^{-p+1}+u(x)^{-p+1})\mu(x,\d y)\, \d x\\
 & = 2 \int_{M_{\lambda r}}\int_{M_{\lambda r}} 
(\tau(y)-\tau(x))^2(u(x)^{-p+1})\mu(x,\d y)\, \d x \\
 & \leq 2\|u^{-p+1}\|_{L^1({M_{\lambda r}})}\sup_{x\in\R^d} \int_{\R^d} 
(\tau(y)-\tau(x))^2\mu(x,\d y) \\
 & \leq c_5r^{-\alam}\left(\sum_{k=1}^d 
\left(\lambda^{\alam/\alpha_k}-1\right)^{-\alpha_k}\right)\|u^{-1}\|_{L^{p-1}(M_
{
\lambda r})}^{p-1}.
\end{align*}
By \autoref{cor:sobolev-local}, we can estimate the left-hand-side of 
\eqref{moserprep} from below
\begin{align*}
 \int_{M_{\lambda r}}&\int_{M_{\lambda r}} 
\left(\tau(x)u(x)^{\frac{-p+1}{2}}-\tau(x)u(x)^{\frac{-p+1}{2}}\right)^2 
\mu(x,\d y)\, \d x \\
& \geq c_6 \|\tau u^{\frac{-p+1}{2}}\|_{L^{\frac{2\beta}{\beta-1}}(M_{r})}^2 - 
r^{-\alam}\left(\sum_{k=1}^d 
(\lambda^{\alam/\alpha_k}-1)^{-\alpha_k}\right)\|\tau 
u^{\frac{-p+1}{2}}\|_{L^{2}(M_{\lambda r})}^2 \\
& \geq c_6 \|u^{\frac{-p+1}{2}}\|_{L^{\frac{2\beta}{\beta-1}}(M_{r})}^2 - 
r^{-\alam}\left(\sum_{k=1}^d 
(\lambda^{\alam/\alpha_k}-1)^{-\alpha_k}\right)\|u^{\frac{-p+1}{2}}\|_{L^{2}(M_{
\lambda r})}^2 \\
& = c_6 \|u^{-1}\|_{L^{(p-1)\frac{\beta}{\beta-1}}(M_{r})}^{p-1} - 
r^{-\alam}\left(\sum_{k=1}^d 
(\lambda^{\alam/\alpha_k}-1)^{-\alpha_k}\right)\|u^{-1}\|_{L^{p-1}(M_{\lambda 
r})}^{p-1}.
\end{align*}
Combining these estimates there exists a constant $c_1>0$, independent of $x_0, 
r, \lambda$, $\alpha_1,\dots,\alpha$ and $u$, but depending on $d$ and 
$2\beta/(\beta-1)$, such that
\begin{align*}
\|u^{-1}\|_{L^{(p-1)\frac{\beta}{\beta-1}}(M_{r})}^{p-1} & \leq 
c_1(\omega^{\frac{-\beta}{q-\beta}}+\frac{p}{p-1})\left(\sum_{k=1}^d 
(\lambda^{\alam/\alpha_k}-1)^{-\alpha_k}\right)r^{-\alam}\|u^{-1}\|_{L^{p-1}(M_{
\lambda 
r})}^{p-1} \\
& \qquad \qquad  + \frac{1}{c_2c_4}\omega 
\|u^{-1}\|_{L^{(p-1)\frac{\beta}{\beta-1}}(M_{\lambda r})}.
\end{align*}
Choosing $\omega$ small enough proves the assertion.
\end{proof}
\begin{lemma}\label{iteration}
 Assume $x_0\in M_1$ and $r\in[0,1)$. Let 
$\lambda\in(1,\min\{r^{-1},\sqrt{2}\})$. Assume $f\in L^{q}(M_{\lambda 
r}(x_0))$ for some $q>\max\{2,\beta\}$
 and let $u\in V^{\mu}(M_{\lambda r}(x_0)\big|\R^d)$ satisfy 
  \begin{align*}
  \mathcal{E}(u,\varphi)&\geq (f,\varphi) \quad \text{for any nonnegative } 
\varphi\in 
H^{\mu}_{M_{\lambda r}}(\R^d),\\
  u(x)&\geq \epsilon \qquad \ \text{ for almost all } x\in M_{\lambda r} \text{ 
and some } \epsilon>\|f\|_{L^q(M_{\lambda r}(x_0))}r^{\alam(q-\beta)/q}.
 \end{align*}
 Then for any $p_0>0$, there is a constant $c_1>0$, independent of 
$u,x_0,\lambda, r,\epsilon$ and $\alpha_1,\dots,\alpha_d$, such that
 \begin{equation}\label{inf}
  \inf\limits_{x\in M_{r}(x_0)} u(x)\geq c_1\left(\dashint_{M_{2r}(x_0)} 
u(x)^{-p_0} \, \d x\right)^{-1/p_0}.
 \end{equation}
\end{lemma}
\begin{proof}
We set $M_r=M_r(x_0)$. For $n\in \N_0$ we define the sequences
\[  r_n=\left(\frac{n+2}{n+1}\right)r \quad \text{and} \quad 
p_n=p_0\left(\frac{\beta}{\beta-1}\right)^n. \]
Then $r_0=2r$, $r_k>r_{k+1}$ for all $k\in\N_0$ and $r_n\searrow r$ as 
$n\to\infty$. Note
\[ r_{n} = \frac{(n+2)^2}{(n+1)(n+3)}r_{n+1} =: \lambda_n r_{n+1}. \]
Moreover $p_0=p_0$, $p_k<p_{k+1}$ for all $k\in \N_0$ and $p_n\nearrow +\infty$ 
as $n\to\infty$.\\
Using 
\[\frac{-\alam}{p_n}-\frac{\alam\beta}{p_{n+1}}=\frac{-\alam\beta}{p_n},\] 
we have
\begin{align*}
\frac{r_{n+1}^{-\alam/p_n}}{|M_{r_{n+1}}|^{1/p_{n+1}}} = \frac{2^{d/(\beta 
p_n)}\lambda_n^{\alam\beta/p_n}}{|M_{r_n}|^{1/p_n}}.
\end{align*}
Moreover, by \autoref{abschfuermoser}, we have for $p=p_n+1$
\begin{align*}
 &\|u^{-1}\|_{L^{p_{n+1}}(M_{r_{n+1}})} = 
\|u^{-1}\|_{L^{p_n\frac{\beta}{\beta-1}}(M_{r_{n+1}})} \\
 \qquad & \leq c_2^{1/p_{n}}\left(\frac{p_{n}+1}{p_n}\right)^{1/p_n} 
\left(\sum_{k=1}^d 
(\lambda_n^{\alam/\alpha_k}-1)^{-\alpha_k}\right)^{1/p_n}r_{n+1}^{-\alam/p_n}
\|u^{-1}\|_
{L^{p_n}(M_{r_n})}. 
\end{align*}
This yields
\begin{align*}
\Bigg(\dashint_{M_{r_{n+1}}} & (u^{-1})^{p_{n+1}} \Bigg)^{1/p_{n+1}} \\
& \leq 2^{d/(\beta 
p_n)}\lambda_n^{\alam\beta/p_n}c_2^{1/p_{n}}\left(\frac{p_{n}+1}{p_n}\right)^{
1/p_n}
 \left(\sum_{k=1}^d (\lambda_n^{\alam/\alpha_k}-1)^{-\alpha_k}\right)^{1/p_n} 
\left(\dashint_{M_{r_{n}}} (u^{-1})^{p_{n}} \right)^{1/p_{n}}.
\end{align*}
which is equivalent to
\begin{equation}\label{iteration1}
\begin{aligned}
 \left(\dashint_{M_{r_{n}}} u^{-p_{n}} \right)^{-1/p_{n}} \leq 2^{d/(\beta 
p_n)}&\lambda_n^{\alam\beta/p_n}c_2^{1/p_{n}}\left(\frac{p_{n}+1}{p_n}\right)^{
1/p_n
} \\
& \times \left(\sum_{k=1}^d 
(\lambda_n^{\alam/\alpha_k}-1)^{-\alpha_k}\right)^{1/p_n} 
\left(\dashint_{M_{r_{n+1}}} u^{-p_{n+1}} \right)^{-1/p_{n+1}}.
\end{aligned}
 \end{equation}
Iterating \eqref{iteration1} leads to
\begin{equation}\label{iteration2}
 \begin{aligned}
 \left(\dashint_{M_{r_{0}}} u^{-p_{0}} \right)^{-1/p_{0}} \leq& 
\left(\prod_{j=0}^n 2^{d/(\beta p_j)}\right)\left(\prod_{j=0}^n 
\lambda_j^{\alam\beta/p_j}\right)\left(\prod_{j=0}^n c_2^{1/p_{j}}\right)
 \left(\prod_{j=0}^{n}\left(\frac{p_{j}+1}{p_j}\right)^{1/p_j}\right) \\
&\quad  \times \left(\prod_{j=0}^n \left(\sum_{k=1}^d 
(\lambda_j^{\alam/\alpha_k}-1)^{-\alpha_k}\right)^{1/p_j}\right) 
\left(\dashint_{M_{r_{n+1}}} u^{-p_{n+1}} \right)^{-1/p_{n+1}}.
\end{aligned}
\end{equation}
One can easily show that the expressions on the right-hand-side of 
\eqref{iteration2} are bounded for $n\to\infty$. 
Since
\[ \lim\limits_{n\to\infty} \left(\dashint_{M_{r_n}} u^{-p_n}\right)^{-1/p_n} = 
\inf_{x\in M_r} u(x), \]
taking the limit $n\to\infty$ in \eqref{iteration2}, proves the assertion.
\end{proof}
From \autoref{iteration} and \autoref{flip-assert} we immediately conclude the 
following result. 
\begin{corollary}\label{weak-harnack-prepare}
 Let $f\in L^q(M_1)$ for some $q>\max\{2,\beta\}$. There are $p_0, c_1>0$  
 such that for every $u\in V^{\mu}(M_1\big|\R^d)$ with $u\geq 0$ in $\R^d$ and
 \[ \mathcal{E}(u,\varphi)\geq (f,\varphi) \quad \text{ for every nonnegative } 
\varphi\in H^{\mu}_{M_1}(\R^d), \]
 the following holds
 \[ \inf\limits_{M_{\frac14}}u \geq c_1\left(\dashint_{M_{\frac12}} 
u(x)^{p_0}\, \d x\right)^{1/p_0} - \|f\|_{L^{q}(M_{\frac{15}{16}})}. \]
\end{corollary}
\begin{proof} This proof follows the lines of \cite[Theorem 4.1]{DyKa15}.
 Define $v=u+\|f\|_{L^{q}(M_{\frac{15}{16}})}.$ Then for any nonnegative 
$\varphi\in H^{\mu}_{M_1}(\R^d)$, one obviously has 
 \[\mathcal{E}(u,\varphi)=\mathcal{E}(v,\varphi).\]
By \autoref{flip-assert} there are a $c_2>0$, $p_0\in(0,1)$ such that
\begin{equation}\label{flip-in-proof}
 \left(\dashint_{M_{\frac12}} v(x)^{p_0} \; \, \d x \right)^{1/p_0} \; \d x 
\leq c_2 \left(\dashint_{M_{\frac12}} v(x)^{-p_0} \; \d x \right)^{-1/p_0}.
\end{equation}
Moreover, by \autoref{iteration} there is a $c_3>0$ such that for $r=\frac12$ 
and $p_0$ as in \eqref{flip-in-proof}
\begin{align*}
 \inf\limits_{x\in M_{\frac14}} v(x) \geq c_3\left(\dashint_{M_{\frac12}} 
v(x)^{-p_0} \, \d x\right)^{-1/p_0} \geq 
\frac{c_3}{c_2}\left(\dashint_{M_{\frac12}} u(x)^{p_0} \, \d x\right)^{1/p_0}.
\end{align*}
which is equivalent to
 \[ \inf\limits_{M_{\frac14}}u \geq c_1\left(\dashint_{M_{\frac12}} 
u(x)^{p_0}\, \d x\right)^{1/p_0} - \|f\|_{L^{q}(M_{\frac{15}{16}})}. \]
\end{proof}
Given $g:\R^d\to\R$, let $g^{+}(x):= \max\{g(x),0\}, g^{-}(x):= 
-\min\{g(x),0\}$.

We have all ingredients in order to prove \autoref{theo:weakharnack}.
\begin{proof}[Proof of \autoref{theo:weakharnack}]
For any nonnegative $\varphi\in H^{\mu}_{M_1}(\R^d)$
 \begin{equation}\label{uplus}
   \mathcal{E}(u^{+},\varphi)=\mathcal{E}(u,\varphi) + 
\mathcal{E}(u^{-},\varphi) \geq 
(f,\varphi) + \mathcal{E}(u^{-},\varphi) .   
 \end{equation}
 Since $\varphi\in H^{\mu}_{M_1}(\R^d)$ and $u^{-}\equiv 0$ on $M_1$ , we have
 \[ (f,\varphi) = \int_{\R^d}f(x)\varphi(x)\, \d x = \int_{M_1}f(x)\varphi(x)\, 
\d x \]
 and 
 \begin{align*}
  \mathcal{E}(u^{-},\varphi) & = \int_{\R^d}\int_{\R^d} 
(u^{-}(y)-u^{-}(x))(\varphi(y)-\varphi(x))\, \mu(x,\d y)\, \d x \\
  & = -2\int_{M_1}\int_{(M_1)^c} u^{-}(y)\varphi(x)\, \mu(x,\d y)\, \d x.
 \end{align*}
 Hence, we get from \eqref{uplus}
  \[\mathcal{E}(u^{+},\varphi)\geq \int_{M_1}\varphi(x) 
\left(f(x)-2\int_{(M_1)^c} 
u^{-}(y)\, \mu(x,\d y) \right)\, \d x.\]
  Therefore, $u^{+}$ satisfies all assumptions of 
\autoref{weak-harnack-prepare} with $q=+\infty$ and $\widetilde{f}:M_1\to\R$, 
defined by 
  \[\widetilde{f}(x)=f(x)-2\int_{\R^d\setminus M_1} u^{-}(y)\, \mu(x,\d y). \]
  If $\sup\limits_{x\in M_{\frac{15}{16}}} \int_{\R^d\setminus M_1} 
u^{-}(z)\mu(x,\d z)=\infty$, then the assertion of the theorem is obviously 
true.
  Thus we can assume this quantity to be finite. Applying 
\autoref{weak-harnack-prepare} and Hölder's inequality
   \begin{align*}
     \inf\limits_{M_{\frac14}}u &\geq c_1\left(\dashint_{M_{\frac12}} 
u(x)^{p_0}\, \d x\right)^{1/p_0} - \|\widetilde{f}\|_{L^{q}(M_{\frac{15}{16}})} 
\\
     & = c_1\left(\dashint_{M_{\frac12}} u(x)^{p_0}\, \d x\right)^{1/p_0} - 
\|f\|_{L^{q}(M_{\frac{15}{16}})} - 2\left\|\int_{\R^d\setminus M_1} u^{-}(y)\, 
\mu(x,\d y)\right\|_{L^{q}(M_{\frac{15}{16}})} \\
     & \geq c_1\left(\dashint_{M_{\frac12}} u(x)^{p_0}\, \d x\right)^{1/p_0} - 
\|f\|_{L^{q}(M_{\frac{15}{16}})} - \sup\limits_{x\in 
M_{\frac{15}{16}}}2\int_{\R^d\setminus M_1} u^{-}(y)\, \mu(x,\d y).
   \end{align*}
\end{proof}
An immediate consequence of \autoref{theo:weakharnack} is the following 
result, which follows via scaling and translation.

\begin{corollary}\label{cor:weak_harnack-scaled}
Let $x_0\in M_1$, $r\in(0,1]$. Let $f\in L^q(M_1(x_0))$ for some 
$q>\max\{2,\beta\}$.  Assume $u\in V^{\mu}(M_r(x_0)\big|\R^d)$ 
satisfies $u\geq 0$ in $M_r(x_0)$ and $\mathcal{E}(u,\varphi)\geq(f,\varphi)$
 for every $\varphi\in H^{\mu}_{M_r(x_0)}(\R^d)$. 
  Then there exists $p_0\in(0,1)$, $c_1>0$, independent of $u,x_0$ and $r$, 
such that
 \begin{align*}
  \inf\limits_{M_{\frac14 r}(x_0)}u \geq c_1\left(\dashint_{M_{\frac12 
r}(x_0)} u(x)^{p_0}\, \d x\right)^{1/p_0} & - r^{\alam}\sup\limits_{x\in 
M_{\frac{15}{16}r}(x_0)} 2\int_{\R^d\setminus M_r(x_0)} u^{-}(z)\mu(x,\d z) \\
& - r^{\alam(1-\frac{\beta}{q})}\|f\|_{L^q}(M_{\frac{15}{16}r}).
\end{align*}
\end{corollary}

\section{H\"older regularity estimates for weak solutions}\label{sec:hoelder}
\allowdisplaybreaks

In this section we prove the main result of this article, i.e. an a 
priori H\"older estimates for weak solutions to $\mathcal{L}u=f$ in $M_1$.  
For this purpose, we first prove a decay of oscillation result. 
We modify the general scheme for the derivation of a 
priori  H\"older estimates developed in \cite{DyKa15}.

\begin{theorem}\label{theo:hoelderprep}
Let $x_0\in\R^d$, $r_0\in(0,1]$. Let $c_a \geq 
1$, $p > 0$ and $\Theta>\lambda>\sigma>1$. Let $f\in L^q(M_1(x_0))$ for some 
$q>\max\{2,\beta\}$.
We assume 
that the weak Harnack inequality holds true in $M_r(x_0)$, i.e. 

\begin{minipage}{0.9\textwidth}
For every $0<r\leq r_0$ and $u\in V^{\mu}(M_r(x_0)\big|\R^d)$ 
satisfying $u\geq 0$ in $M_r(x_0)$ and $\mathcal{E}(u,\varphi)=(f,\varphi)$
 for every $\varphi\in H^{\mu}_{M_r(x_0)}(\R^d)$,
 \begin{align}\label{theo:hoelderprep_assumption}
\begin{split}
  \left(\dashint_{M_\frac{r}{\lambda}(x_0)}u(x)^p \, \d x \right)^{1/p} \leq 
c_a \Bigg( \inf\limits_{M_{\frac{r}{\Theta}}(x_0)} u + 
 r^{\alam}\sup\limits_{x\in M_{\frac{r}{\sigma}}(x_0)} \int_{\R^d} 
u^{-}(z)\mu(x,\d z)  \\
\qquad + r^{\alam(1-\frac{\beta}{q})} 
\|f\|_{L^q(M_{\frac{r}{\sigma}})} \Bigg)\,.
\end{split}
 \end{align}
\end{minipage}

Then there exists $\delta\in(0,1)$, $c > 0$ such that for $r\in(0,r_0]$, 
 $u\in V^{\mu}(M_r(x_0)\big|\R^d)$ 
satisfying $u\geq 0$ in $M_r(x_0)$ and $\mathcal{E}(u,\varphi)=(f,\varphi)$
 for every $\varphi\in H^{\mu}_{M_r(x_0)}(\R^d)$,
\begin{equation}\label{decayofosc}
 \begin{split} 
 \osc\limits_{M_{\rho}(x_0)} u \leq 2\Theta^{\delta} 
\|u\|_{\infty} \left(\frac{\rho}{r}\right)^{\delta} + c \Theta^{\delta} 
\left(\frac{\rho}{r}\right)^{\delta} r^{\alam(1-\frac{\beta}{q})} 
\|f\|_{L^q(M_{\frac{r}{\sigma}})}, \qquad (0<\rho\leq r)\,.
\end{split}
\end{equation}
\end{theorem}

\begin{proof}
The strategy of the proof is well-known and can be traced back to G.~A. 
Harnack himself. We adapt the proof of \cite[Theorem 1.4]{DyKa15} to the 
anisotropic setting. We also include a right-hand side function $f$. In the 
following, we write $M_r$ instead of $M_r(x_0)$ for $r>0$. 

Let $c_a$ and $p$ be the constants from \eqref{theo:hoelderprep_assumption}. 
Set $\kappa=(2c_a2^{1/p})^{-1}$
 and
 \begin{equation}\label{kappa}
 \delta=\frac{\log\left(\frac{2}{2-\kappa}\right)}{\log(\Theta)} \quad 
\Longrightarrow 1-\frac{\kappa}{2}=\Theta^{-\delta}.   
 \end{equation}

Assume $0<r\leq r_0$ and $u\in V^{\mu}(M_r(x_0)\big|\R^d)$ 
satisfies $u\geq 0$ in $M_r(x_0)$ and $\mathcal{E}(u,\varphi)=(f,\varphi)$
 for every $\varphi\in H^{\mu}_{M_r(x_0)}(\R^d)$. Set 
\begin{align*}
\wtu(x) = u(x) \Big[ \|u\|_{\infty} + \tfrac{2}{\kappa}
r^{\alam(1-\frac{\beta}{q})} \|f\|_{L^q(M_{\frac{r}{\sigma}})} 
 \Big]^{-1}
\end{align*}

 Set $b_0=\|\wtu\|_{\infty}$, $a_0=\inf\{\wtu (x)\colon x\in\R^d\}$ and 
$b_{-n}=b_0, a_{-n}=a_0$ for $n\in\N$. Our aim is to construct an increasing
 sequence $(a_n)_{n\in\Z}$ and a decreasing sequence $(b_n)_{n\in\Z}$ such that 
for all $n\in\Z$
 \begin{equation}\label{Mm}
   \begin{cases}
  a_n\leq \wtu(z) \leq b_n \\
  b_n-a_n\leq 2\Theta^{-n\delta}
 \end{cases}
 \end{equation}
for almost all $z\in M_{r\Theta^{-n}}$. Before we prove \eqref{Mm}, we show that
\eqref{Mm} implies the assertion. Let $\rho\in(0,r]$. There is $j\in \N_0$ such 
that $r\Theta^{-j-1} \leq \rho \leq r\Theta^{-j}.$
Note, that this implies in particular $\Theta^{-j} \leq \rho\Theta/r.$ 
From \eqref{Mm}, we deduce
\begin{align*}
 \osc\limits_{M_{\rho}} \wtu &\leq \osc\limits_{M_{r\Theta^{-j}}} \wtu \leq 
b_j-a_j \leq 2\Theta^{-\delta j} \leq 
2\Theta^{\delta}\left(\frac{\rho}{r}\right)^{\delta} \,,
\end{align*}
where from the assertion follows. It remains to show \eqref{Mm}. 

Assume 
there is $k\in\N$ and there are $b_n,a_n$,
such that \eqref{Mm} holds true for $n\leq k-1$. We need to choose $b_k,a_k$ 
such that \eqref{Mm} still holds for $n=k$.
For $z\in\R^d$ set 
\[ 
v(z)=\left(\wtu(z)-\frac{b_{k-1}+a_{k-1}}{2}\right) \Theta^{(k-1)\delta} \,. 
\]
Then $|v(z)|\leq 1$ for almost every $z\in 
M_{r\Theta^{-(k-1)}}$ and $\mathcal{E}(v,\varphi)=(\widetilde{f},\varphi)$ for 
every $\varphi\in H_{M_{r\Theta^{-(k-1)}}}^{\mu}(\R^d)$, where
\begin{equation}\label{widetildef}
\widetilde{f}(x) = \frac{\Theta^{(k-1)\delta}}{\|u\|_{\infty} + 
\tfrac{2}{\kappa}
r^{\alam(1-\frac{\beta}{q})} \|f\|_{L^q(M_{\frac{r}{\sigma}})}} f(x).
\end{equation}  
Let $z\in\R^d$ be such that $z\notin M_{r\Theta^{-k+1}}$.
Choose $j\in\N$ such that $z\in M_{r\Theta^{-k+j+1}}\setminus 
M_{r\Theta^{-k+j}}$.
 For such $z$ and $j$, we conclude
 \begin{align*}
  \frac{v(z)}{\Theta^{(k-1)\delta}}  & \geq a_{k-j-1} - 
\frac{b_{k-1}+a_{k-1}}{2} \\
  & \geq -(b_{k-j-1}-a_{k-j-1}) + \frac{b_{k-1}-a_{k-1}}{2} \\
  & \geq -2\Theta^{-(k-j-1)\delta} + \frac{b_{k-1}-a_{k-1}}{2}.
 \end{align*}
 Thus 
\begin{equation} \label{v1} 
v(z) \geq 1-2\Theta^{j\delta} 
\end{equation}
and similarly
\begin{equation} \label{v2} 
v(z) \leq 2\Theta^{j\delta}-1 
\end{equation}
for $z\in M_{r\Theta^{-k+j+1}}\setminus M_{r\Theta^{-k+j}}$.
We will distinguish two cases.
\begin{enumerate}
 \item First assume 
 \begin{equation}\label{case1}
  |\{x\in M_{\frac{r\Theta^{-k+1}}{\lambda}} \colon v(x) \leq 0\}|\geq 
\frac{1}{2}|M_{\frac{r\Theta^{-k+1}}{\lambda}}|.
 \end{equation}
 Our aim is to show that in this case 
 \begin{equation}\label{vkleinerkappa}
v(z)\leq 1-\kappa \quad \text{ for almost every } z\in M_{r\Theta^{-k}}.  
 \end{equation} 
 We will first show that this implies \eqref{Mm}. 
 Recall, that \eqref{Mm} holds true for $n\leq k-1$. Hence we need to find 
$a_k, b_k$ satisfying \eqref{Mm}.
 Assume \eqref{vkleinerkappa} holds.
 
Then for almost any $z\in M_{r\Theta^{-k}}$
 \begin{align*}
  \wtu (z) &= \frac{1}{\Theta^{(k-1)\delta}}v(z) + \frac{b_{k-1}+a_{k-1}}{2}\\
  & \leq \frac{1}{\Theta^{(k-1)\delta}}(1-\kappa) + \frac{b_{k-1}+a_{k-1}}{2} 
\\
  & \leq a_{k-1} + \left(1-\frac{\kappa}{2}\right) 2\Theta^{-(k-1)\delta} \\
  & \leq a_{k-1} + 2\Theta^{-k\delta}.
 \end{align*}
If we now set $a_k=a_{k-1}$ and $b_k=b_k+2\Theta^{-k\delta}$, then by the 
induction hypothesis $u(z)\geq a_{k-1}=a_k$ and by the
previous calculation $u(z)\leq b_k$. Hence \eqref{Mm} follows. 

It remains to prove $v(z)\leq 1-\kappa$ for almost every
$z\in M_{r\Theta^{-k}}$. Consider $w=1-v$ and note $w\geq 0$ in 
$M_{r\Theta^{-(k-1)}}$ and $\mathcal{E}(w,\varphi)=(\widetilde{f},\varphi)$ 
for every $\varphi\in H_{M_{r\Theta^{-(k-1)}}}^{\mu}(\R^d)$, where 
$\widetilde{f}$ is defined as in \eqref{widetildef}.
We apply the weak Harnack inequality \eqref{theo:hoelderprep_assumption} to the 
function $w$ for $r_1=r\Theta^{-k+1}\in(0,r]$. Then

\begin{align*}
  \Bigg(\dashint_{M_\frac{r_1}{\lambda}}w(x)^p \, \d x \Bigg)^{1/p} \leq c_a 
\Bigg(\inf\limits_{M_{\frac{r_1}{\Theta}}}& w + 
 r_1^{\alam}\sup\limits_{x\in M_{\frac{r_1}{\sigma}}} \int_{\R^d} 
w^{-}(z)\mu(x,\d z) \\
& + \frac{r_1^{\alam(1-\frac{\beta}{q})} 
\|f\|_{L^q(M_{\frac{r_1}{\sigma}})}\Theta^{(k-1)\delta}}{\|u\|_{\infty} + 
\tfrac{2}{\kappa}
r^{\alam(1-\frac{\beta}{q})} \|f\|_{L^q(M_{\frac{r}{\sigma}})}}\bigg).
\end{align*}

We assume $\delta\leq \alam(1-\frac{\beta}{q})$. Then

\begin{align*}
\frac{r_1^{\alam(1-\frac{\beta}{q})} 
\|f\|_{L^q(M_{\frac{r_1}{\sigma}})}\Theta^{(k-1)\delta}}{\|u\|_{\infty} + 
\tfrac{2}{\kappa}
r^{\alam(1-\frac{\beta}{q})} \|f\|_{L^q(M_{\frac{r}{\sigma}})}} 
& \leq \frac{ 
\|f\|_{L^q(M_{\frac{r_1}{\sigma}})}\Theta^{(k-1)(\delta-\alam(1-\frac{\beta}{q}
))}}{\tfrac{2}{\kappa}
 \|f\|_{L^q(M_{\frac{r_1}{\sigma}})}} \leq \frac{\kappa}{2}.
\end{align*}

Using assumption \eqref{case1} the left hand side can be estimated as follows
\begin{align*}
 \Bigg(\dashint_{M_\frac{r\Theta^{-(k-1)}}{\lambda}}w(x)^p \, \d x \Bigg)^{1/p} 
&\geq 
\Bigg(\dashint_{M_\frac{r\Theta^{-(k-1)}}{\lambda}}w(x)^p\mathds{1}_{\{v(x)\leq 
0\}} \, \d x \Bigg)^{1/p} \\
 & = \Bigg(\frac{|\{x\in M_{\frac{r\Theta^{-k+1}}{\lambda}} \colon v(x) \leq 
0\}|}{|M_{\frac{r\Theta^{-k+1}}{\lambda}}|}\Bigg)^{1/p} \\
 & \geq 
\Bigg(\frac{\frac{1}{2}|M_{\frac{r\Theta^{-k+1}}{\lambda}}|}{|M_{\frac{r\Theta^{
-k+1}}{\lambda}}|}\Bigg)^{1/p} = \frac{1}{2^{1/p}}.
\end{align*}

Moreover by \eqref{v2}
\begin{equation} \label{1-v}
(1-v(z))^{-} \leq (1-2\Theta^{j\delta}+1)^{-}=2\Theta^{j\delta}-2. 
\end{equation}

Consequently

\begin{align*}
 \inf\limits_{M_{r\Theta^{-k}}} w  \geq  2\kappa - \frac{\kappa}{2}  
 - (r\Theta^{-(k-1)})^{\alam} \sup\limits_{x\in 
M_{\frac{r\Theta^{-(k-1)}}{\sigma}}} \int_{\R^d} 
w^{-}(z) \mu(x,\d z)
\end{align*}

Let us show that the last term depends continuously on $\delta$ and can be 
made arbitrarily small. Note, that $w\geq 0$ in 
$M_{r\Theta^{-(k-1)}}$. Let $x \in M_{\frac{r\Theta^{-(k-1)}}{\sigma}}$ and $j 
\in \N$. From \eqref{1-v} we deduce, 
\begin{align*}
 &\int_{\R^d} w^{-}(z)\mu(x,\d z) =\int_{\R^d\setminus 
M_{r\Theta^{-(k-1)}}} w^{-}(z)\mu(x,\d z) \\ 
& \qquad  = \sum_{j=1}^\infty \int_{M_{r\Theta^{-k+j+1}}\setminus 
M_{r\Theta^{-k+j}}} \mu(x,\d z)  
\leq \sum_{j=1}^{\infty} (2\Theta^{j\delta}-2) \int_{\R^d \setminus 
M_{r\Theta^{-k+j}} } \mu(x,\d z). 
\end{align*}
Note that
\begin{align*}
& (r\Theta^{-(k-1)})^{\alam}\mu(x,\R^d\setminus M_{r\Theta^{-k+j}}) \\
& \leq (r\Theta^{-(k-1)})^{\alam} 2\sum_{i=1}^d 
(\alpha_i)(2-\alpha_i)\int_{(r\Theta^{-k+j})^{\alam/\alpha_i} 
((\sigma-1)/\sigma)}^{\infty} |h|^{-1-\alpha_i} \, \d h \\
& = (r\Theta^{-(k-1)})^{\alam} 2\sum_{i=1}^d (2-\alpha_i) 
(r\Theta^{-k+j})^{-\alam}(\sigma/(\sigma-1))^{\alpha_i} \\
& \leq 4d (\sigma/(\sigma-1))^{\alam}(\Theta^{-j+1})^{\alam} \,.
\end{align*}
Thus for every $l\in\N$, 
\begin{align*}
r_1^{\alam}\int_{\R^d} w^{-}(z)\mu(x,\d z) &\leq c_1  \sum_{j=1}^l 
(\Theta^{j\delta}-1) (\Theta^{-j+1})^{\alam} +  
 c_1  \sum_{j=l+1}^{\infty} \Theta^{j\delta} (\Theta^{-j+1})^{\alam} \\
 & =: I_1 + I_2\,.
\end{align*}
with a positive constant $c_1$ depending only on $\sigma, d, \alam$. 
From now on, we assume $\delta \leq \frac{\alpha}{2}$. First, we choose $l \in 
\N$ sufficiently large in dependence of $\alam$ such that $I_2
\leq \frac{\kappa}{4}$. Second, we choose 
$\delta$ sufficiently small such that $I_1 \leq 
\frac{\kappa}{4}$. Since these choices are independent of $x$ and $k$, we 
have proved
\begin{align*}
r_1^{\alam}\sup\limits_{x\in M_{\frac{r_1}{\sigma}}} \int_{\R^d} 
w^{-}(z)\mu(x,\d z) \leq \frac{\kappa}{2} \,.
\end{align*}

Thus 
\[ w\geq \inf\limits_{M_{r\Theta^{-k}}} w \geq \kappa \quad \text{on } 
M_{r\Theta^{-k}},\]
or equivalently $v\leq 1-\kappa$ on $M_{r\Theta^{-k}}$.
\item  Now, we assume 
 \begin{equation}\label{case2}
  |\{x\in M_{\frac{r\Theta^{-k+1}}{\lambda}} \colon v(x) > 0\}|\geq 
\frac{1}{2}|M_{\frac{r\Theta^{-k+1}}{\lambda}}|.
 \end{equation}
 Our aim is to show that in this case 
 \[ v(z)\geq -1+\kappa \quad \text{ for almost every } M_{r\Theta^{-k}}.\]
 
 Similar to the first case, this implies for almost every $z\in 
M_{r\Theta^{-k}}$
  \begin{align*}
  \wtu(z) \geq b_{k-1} - 2\Theta^{-k\delta}.
 \end{align*}
Choosing $b_k=b_{k-1}$ and $a_k=a_{k-1}-2\Theta^{-k\delta}$, then by the 
induction hypothesis $\wtu(z)\leq b_{k-1}=b_k$ and by the
previous calculation $\wtu(z)\geq a_k$. Hence \eqref{Mm} follows. 

It remains to show in this case $v(z)\leq -1+\kappa$ for almost every
$z\in M_{r\Theta^{-k}}$. 

Consider $w=1+v$ and note $\mathcal{E}(v,\varphi)=(\widetilde{f},\varphi)$ for 
every $\varphi\in H_{M_{r\Theta^{-(k-1)}}}^{\mu}(\R^d)$ and $w\geq 
0$ in $M_{r\Theta^{-(k-1)}}$.
Then the desired statement follows analogously to Case 1.
\end{enumerate}
\end{proof}

Finally, we can prove our main result concerning Hölder regularity estimates.

\begin{proof}[Proof of \autoref{theo:hoelder}]

We distinguish two following two cases:\\
If $\mathbbm{d}(x,y) \geq \frac14$, then \eqref{Hoelder-estimate} follows from \autoref{cor:weak_harnack-scaled} and 
\autoref{theo:hoelderprep} and the observation that for all $x,y\in M_{\frac{1}{2}}$
 \begin{equation}\label{dabschhoeld}
 \begin{aligned}
   \mathbbm{d}(x,y) & =  
\sup\limits_{k\in\{1,\dots,d\}}\left\{|x_k-y_k|^{\alpha_k/\alam}\right\} \leq 
\sup\limits_{k\in\{1,\dots,d\}}\left\{|x_k-y_k|^{\alpha_{\min}/\alam}\right\} \\
   & = 
\left(\sup\limits_{k\in\{1,\dots,d\}}\left\{|x_k-y_k|\right\}\right)^{\alpha_{
\min}/\alam} \leq  |x-y|^{\alpha_{\min}/\alam}.
 \end{aligned}
 \end{equation}
 If $\mathbbm{d}(x,y) <\frac14$, then there is a $\rho\in(0,\frac14)$ such 
 that $\frac{\rho}{2}\leq \mathbbm{d}(x,y) \leq \rho$. 
 We cover $M_{1-4\rho}$ by a countable family of balls $(M_i)_{i}$ with respect 
 to the metric space $(\R^d,\mathbbm{d})$ with radii $\rho$, such that there
 is a $j$ with $x,y\in 2M_j$, where $2M_j$ is the ball with the same center as 
 $M_j$ but with radius $2\rho$.  Let $\widetilde{M_j}$ be the ball with the same 
 center as $M_j$ and maximal radius such that $\widetilde{M_j}\subset M_1$. By
\autoref{theo:hoelderprep} there is a $\delta_1\in(0,1)$ and $c_2>0$ such that
\begin{align*}
\osc\limits_{2M_j} u \leq c_2 \rho^{\delta} (\|u\|_{\infty} 
+ \|f\|_{L^q(M_{\frac{15}{16}})}) 
\leq c_3 |x-y|^{\delta_1\alpha_{\min}/\alpha_{\max}} 
 (\|u\|_{\infty} + \|f\|_{L^q(M_{\frac{15}{16}})}),
\end{align*}
which finishes the proof.
\end{proof}

\newpage
\bibliographystyle{abbrv}
\bibliography{bib}
\end{document}